\documentclass[11pt,reqno]{amsart}
\usepackage[utf8]{inputenc}
\usepackage[T1]{fontenc}
\usepackage[english]{babel}
\usepackage[usenames,dvipsnames]{xcolor}
\definecolor{darkred}{rgb}{0.7,0,0}
\definecolor{darkblue}{rgb}{0,0,0.7}
\usepackage[colorlinks=true,urlcolor=darkblue,linkcolor=darkred,citecolor=darkblue]{hyperref}
\usepackage{amsfonts}
\usepackage{amssymb}
\usepackage{amsmath}
\usepackage{stmaryrd}
\usepackage{dsfont}
\usepackage{amsthm}
\usepackage[left=3cm, right=3cm, top=2.2cm,bottom=2.2cm]{geometry}
\usepackage{parskip}
\usepackage{mathtools}
\usepackage{tikz}
\usepackage{tikz-cd}

\patchcmd{\section}{\scshape}{\bfseries}{}{}
\makeatletter
\renewcommand{\@secnumfont}{\bfseries}
\makeatother

\usepackage{cleveref}
\crefname{thm}{Theorem}{Theorems}
\crefname{prop}{Proposition}{Propositions}
\crefname{thmintro}{Theorem}{Theorems}
\crefname{propintro}{Proposition}{Propositions}
\crefname{lemma}{Lemma}{Lemmas}
\crefname{rem}{Remark}{Remarks}
\crefname{cor}{Corollary}{Corollaries}
\crefname{defi}{Definition}{Definitions}
\crefname{ex}{Example}{Examples}
\crefname{section}{Section}{Sections}

\newtheoremstyle{standard}{9pt}{9pt}{\itshape}{}{\bfseries}{.}{.5em}{}
\theoremstyle{standard}
\newtheorem{lemma}{Lemma}[section]
\newtheorem{thm}[lemma]{Theorem}
\newtheorem{prop}[lemma]{Proposition}
\newtheorem{cor}[lemma]{Corollary}

\newtheorem{thmintro}{Theorem}[section]

\newtheorem{propintro}[thmintro]{Proposition}

\newtheoremstyle{definition}{9pt}{9pt}{}{}{\bfseries}{.}{.5em}{}    
\theoremstyle{definition}
\newtheorem{defi}[lemma]{Definition}  
\newtheorem{rem}[lemma]{Remark}
\newtheorem{ex}[lemma]{Example}

\let\originalleft\left
\let\originalright\right
\renewcommand{\left}{\mathopen{}\mathclose\bgroup\originalleft}
\renewcommand{\right}{\aftergroup\egroup\originalright}

\newcommand{\A}{\mathcal{A}}
\newcommand{\B}{\mathcal{B}}
\newcommand{\C}{\mathcal{C}}
\newcommand{\D}{\mathcal{D}}
\newcommand{\E}{\mathcal{E}}
\newcommand{\F}{\mathcal{F}}

\newcommand{\I}{\mathcal{I}}

\newcommand{\K}{\mathcal{K}}

\renewcommand{\L}{\mathcal{L}}

\renewcommand{\O}{\mathcal{O}}

\newcommand{\Q}{\mathcal{Q}}

\renewcommand{\S}{\mathcal{S}}
\newcommand{\T}{\mathcal{T}}

\newcommand{\IN}{\mathds{N}}
\newcommand{\IZ}{\mathds{Z}}

\newcommand{\IS}{\mathbb{S}}

\DeclareMathOperator{\Top}{\mathsf{Top}}
\DeclareMathOperator{\Topf}{\Top^{\mathsf{f}}}
\DeclareMathOperator{\BiTop}{\mathsf{BiTop}}
\DeclareMathOperator{\CoTop}{\mathsf{CoTop}}
\DeclareMathOperator{\Haus}{\mathsf{Haus}}
\DeclareMathOperator{\CoHaus}{\mathsf{CoHaus}}
\DeclareMathOperator{\Set}{\mathsf{Set}}
\DeclareMathOperator{\Grp}{\mathsf{Grp}}
\DeclareMathOperator{\Qcoh}{\mathsf{Qcoh}}
\DeclareMathOperator{\Ab}{\mathsf{Ab}}
\DeclareMathOperator{\CoGrp}{\mathsf{CoGrp}}
\DeclareMathOperator{\Pos}{\mathsf{Pos}}
\DeclareMathOperator{\CoPos}{\mathsf{CoPos}}
\DeclareMathOperator{\Pre}{\mathsf{Pre}}
\DeclareMathOperator{\CoPre}{\mathsf{CoPre}}
\DeclareMathOperator{\Alg}{\mathsf{Alg}}
\DeclareMathOperator{\Mono}{\mathsf{Mono}}

\DeclareMathOperator{\CABA}{\mathsf{CABA}}
\DeclareMathOperator{\CoCABA}{\mathsf{CoCABA}}
\DeclareMathOperator{\Sh}{\mathsf{Sh}}
\DeclareMathOperator{\Rel}{\mathsf{Rel}}
\DeclareMathOperator{\Frm}{\mathsf{Frm}}
\DeclareMathOperator{\CoSh}{\mathsf{CoSh}}
\DeclareMathOperator{\Eq}{\mathsf{Eq}}
\DeclareMathOperator{\SupLat}{\mathsf{SupLat}}
\DeclareMathOperator{\Mod}{\mathsf{Mod}}
\DeclareMathOperator{\CoAlg}{\mathsf{CoAlg}}

\DeclareMathOperator{\CompHaus}{\mathsf{CompHaus}}
\DeclareMathOperator{\Cat}{\mathsf{Cat}}
\DeclareMathOperator{\CoCat}{\mathsf{CoCat}}
\DeclareMathOperator{\Meas}{\mathsf{Meas}}
\DeclareMathOperator{\Monad}{\mathsf{Monad}}
\DeclareMathOperator{\Law}{\mathsf{Law}}
\DeclareMathOperator{\Gra}{\mathsf{Gra}}

\DeclareMathOperator{\op}{op}
\DeclareMathOperator*{\colim}{colim}

\DeclareMathOperator{\Hom}{Hom}

\newcommand{\pr}{\mathrm{pr}}

\DeclareMathOperator{\id}{id}
\DeclareMathOperator{\ev}{ev}
\renewcommand{\c}{\mathrm{c}}
\newcommand{\wc}{\mathrm{wc}}
\newcommand{\s}{\mathrm{s}}

\begin{document}


\author{Martin Brandenburg}
\title{Large limit sketches and topological space objects}
\thanks{\emph{E-mail address:} \texttt{brandenburg@uni-muenster.de}}
\date{\today}


\begin{abstract}
For a (possibly large) realized limit sketch $\S$ such that every $\S$-model is \emph{small} in a suitable sense we show that the category of cocontinuous functors $\Mod(\S) \to \C$ into a cocomplete category $\C$ is equivalent to the category $\Mod_{\C}(\S^{\op})$ of $\C$-valued $\S^{\op}$-models. From this result we deduce universal properties of several examples of cocomplete categories appearing in practice. It can be applied in particular to infinitary Lawvere theories, generalizing the well-known case of finitary Lawvere theories.\\We also look at a large limit sketch that models $\Top$, study the corresponding notion of an internal \emph{net-based topological space object}, and deduce from our main result that cocontinuous functors $\Top \to \C$ into a cocomplete category $\C$ correspond to \emph{net-based cotopological space objects} internal to $\C$. Finally, we describe a limit sketch that models $\Top^{\op}$ and deduce from our main result that continuous functors $\Top \to \C$ into a complete category $\C$ correspond to \emph{frame-based topological space objects} internal to $\C$. Thus, we characterize $\Top$ both as a cocomplete and as a complete category. Thereby we get two new conceptual proofs of Isbell's classification of cocontinuous functors $\Top \to \Top$ in terms of topological topologies.
\end{abstract}


\maketitle


\section{Introduction}

Since Freyd's seminal paper on algebra-valued functors \cite{freyd} there has been an interest to classify certain functors between algebraic categories in terms of coalgebraic structures. For example, when $\C$ is a cocomplete category, a right adjoint functor $\C \to \Grp$ corresponds to a cogroup internal to $\C$. In fact, a cogroup structure on an object $X \in \C$ endows the sets $\Hom(X,Y)$ with a group structure, so that we get a functor $\Hom(X,-) : \C \to \Grp$. For example, the identity functor $\Grp \to \Grp$ is represented by the free group on one generator with its canonical cogroup structure. A right adjoint functor $\C \to \Grp$ can equivalently be described by a cocontinuous functor
\[\Grp \to \C,\]
its left adjoint. Hence, $\Grp$ is the universal example of a cocomplete category with an internal cogroup object. We can see this as a universal property and hence a characterization of $\Grp$ inside the $2$-category of cocomplete categories.

The classical Eilenberg-Watts Theorem \cite{eilenberg} has a similar flavor. It establishes an equivalence between cocontinuous functors ${}_R{\Mod} \to {}_S{\Mod}$ and $(S,R)$-bimodules. More generally, for any preadditive category $\C$ cocontinuous functors ${}_R{\Mod} \to \C$ correspond to $R$-right modules internal to $\C$. This is a $2$-categorical characterization of ${}_R{\Mod}$.

However, there is nothing special about groups and modules here, as it turns out that for every finitary algebraic theory (i.e.\ Lawvere theory) $\L$ and every cocomplete category $\C$ the cocontinuous functors
\[\Mod(\L) \to \C\]
correspond to $\L$-coalgebras internal to $\C$, see \cite[Theorem 13]{poinsot}. Namely, any $\L$-coalgebra $N$ induces a tensor product functor $N \otimes_{\L} - : \Mod(\L) \to \Set$ that is left adjoint to a suitable representable functor $\Hom(N,-) : \C \to \Mod(\L)$. The main goal of this paper is to generalize this result to infinitary Lawvere theories, and in fact to arbitrary limit sketches. This allows us to classify cocontinuous functors not just on finitary algebraic categories, but also on infinitary algebraic categories such as $\SupLat$, $\CompHaus$ and on (infinitary) essentially algebraic categories such as $\Pos$, $\Top$, $\Top^{\op}$.
 
Sketches were introduced and studied by Ehresmann and his school \cite{ehresmann1,ehresmann2,ehresmann3} in order to specify a species of mathematical structure. They provide a syntactic representation of various types of structured categories \cite{barrwells}. We refer to \cite{wells} for a summary of their theory and many more references to the literature. Briefly, a sketch $(\E,\S)$ is defined as a category $\E$ with a class $\S$ of distinguished cones and cocones in it. It is called small when $\E$ is small and $\S$ is a set.

Sketches are of interest since many categories are equivalent to the category $\Mod(\S)$ of models of a sketch $\S$. Here, models are functors $\E \to \Set$ that map the distinguished (co)cones to (co)limit (co)cones in $\Set$. We may also replace $\Set$ by any category $\C$ to obtain the notion of a $\C$-valued model. Specifically, Lair has shown that a category is the category of models of a small sketch if and only it is accessible \cite{lair}, and Gabriel-Ulmer have shown that it is the category of models of a small limit sketch (i.e.\ with cones only) if and only it is locally presentable \cite{gabrielulmer, adamekrosicky}.

In particular, lots of categories encountered in practice can be described by small sketches, and large sketches can model even more categories \cite{guitart}. The idea to represent specific types of categories as categories of structure-preserving functors goes back to Lawvere's thesis \cite{lawvere}, and in fact Lawvere theories and their algebras are special instances of small limit sketches and their models.

For a \emph{small} limit sketch $(\E,\S)$ it is already known that cocontinuous functors
\[\Mod(\S) \to \C\]
correspond to $\S^{\op}$-models internal to $\C$, where $\S^{\op}$ is the corresponding colimit sketch in $\E^{\op}$. This has been shown by Pultr \cite[Theorem 2.5]{pultr}, a more recent reference is \cite[Theorem 2.2.4]{chirvasitu}. The proof relies on the universal property of the cocompletion $\Hom(\E,\Set)$ of the small category $\E^{\op}$ as well as the existence of a reflection $\Hom(\E,\Set) \to \Mod(\S)$. The latter is not available for large sketches (in fact, $\Mod(\S)$ may fail to be cocomplete), so that the proof needs to be modified.

\textbf{Main results.} We recall the basic theory of sketches in \cref{sec:limit}. In \cref{sec:small} we introduce the notion of a \emph{small model}, which is then used in \cref{sec:univ} to show our main result.
 
\begin{thmintro}[\cref{main}] \label{mainintro}
Let $\S$ be a (possibly large) realized limit sketch such that every $\S$-model is small. Then, for any cocomplete category $\C$ there is a natural equivalence of categories
\[\Hom_{\c}(\Mod(\S),\C) \simeq \Mod_{\C}(\S^{\op}).\]
Thus, if $\Mod(\S)$ is also cocomplete, $\Mod(\S)$ is the universal example of a cocomplete category with a model of the colimit sketch $\S^{\op}$, namely the Yoneda model $Y$ with $Y(A) \coloneqq \Hom(A,-)$.
\end{thmintro}

Here, $\Hom_\c$ denotes the category of cocontinuous functors. Thus, \cref{mainintro} gives a universal property of $\Mod(\S)$ as a cocomplete category. This can be applied to lots of specific examples, which we showcase in \cref{sec:examples}. We also prove a more precise result when not every $\S$-model is assumed to be small (see \cref{maintechnical}). For the empty sketch this yields the well-known free cocompletion of a category (locally small, but not necessarily small).
 
Notice that \cref{mainintro} equivalently says that (under all the mentioned assumptions) continuous functors $\Mod(\S)^{\op} \to \D$ into a complete category $\D$ correspond to $\D$-valued models of $\S$, which means that $\Mod(\S)^{\op}$ contains the \emph{generic model} of $\S$ as defined in \cite[Section 4.3]{barrwells}. Thus, we generalize the existence of generic models from the small case to a more general case.

In \cref{mainintro} it turns out that every cocontinuous functor $\Mod(\S) \to \C$ is already a left adjoint. This motivates the notion of a \emph{strongly compact category}, studied in \cref{sec:compact}, which is a variant of Isbell's notion of a compact category \cite{isbell1}: every cocontinuous functor on such a category must be a left adjoint. We have the following classification.
 
\begin{propintro}[\cref{modiscompact}] \label{compchar}
A category is strongly compact if and only if it is the category of models of a realized limit sketch such that every model is small.
\end{propintro}
  
We specialize \cref{mainintro} to infinitary Lawvere theories as well, whose basic theory we recall in the expository \cref{sec:lawapp}, in particular their equivalence to monads on $\Set$. This yields the following result that extends the already known finitary case.

\begin{thmintro}[\cref{lawue}]
Let $\L$ be an infinitary Lawvere theory. Then $\Mod(\L)$ is cocomplete, and for every cocomplete category $\C$ we have
\[\Hom_{\c}(\Mod(\L),\C) \simeq \Mod_{\C}(\L^{\op}).\]
\end{thmintro}

We state this universal property also in terms of monads (see \cref{monue}) and apply it to several examples in \cref{sec:examples}. For example, $\Set^{\op} \simeq \CABA$ is the universal example of a cocomplete category with an internal co(atomic boolean algebra), $\SupLat$ is the universal example of a cocomplete category with an internal cosuplattice, and $\CompHaus$ is the universal example of a cocomplete category with an internal co(compact Hausdorff space).
 
Whereas \cref{mainintro} is \emph{abstract nonsense}, the classification of internal $\S^{\op}$-models for specific sketches $\S$ is a non-trivial task. For a basic example, the category of preorders $\Pre$ is the universal example of a cocomplete category with an internal copreorder. The classification of copreorders in $\Set$ by Lumsdaine \cite{lumsdaine} now yields $\Hom_{\c}(\Pre,\Set) \simeq \Mono(\Set)$ and $\Hom_{\c}(\Pos,\Set)  \simeq \Set$.

But examples are not limited to (finitary or infinitary) algebraic categories. In fact, topological spaces also enter the scene because of the following result, proven in \cref{sec:topsketch}.
  
\begin{thmintro}[\cref{topsketchthm,topsmall}]\label{kelley}
The characterization of topological spaces in terms of net convergence produces a large realized limit sketch $\T$ with
\[\Top \simeq \Mod(\T).\]
Moreover, every $\T$-model is small.
\end{thmintro}

That there is \emph{some} realized limit sketch for $\Top$ is already proven quickly in \cref{sec:compact} using the theory of strongly compact categories. The whole point of \cref{kelley} is that we can find a specific, not so large, and hence useful limit sketch for topological spaces. It allows us to think of topological spaces as infinitary essentially algebraic objects. Our limit sketch is very similar to Edgar's characterization of topological spaces \cite{edgar} in terms of net convergence. One of the main ingredients is the observation is that a convergent net indexed by a directed set $P$ is the same as a continuous map on a suitably defined topological space $P \cup \{\infty\}$.

Similar, but more complex descriptions of $\Top$ have been obtained before. Burroni \cite{burroni1,burroni2} used filter convergence to define a mixed sketch (cones and cocones) that models $\Top$. Barr \cite{barr} described $\Top$ as the category of relational $\beta$-modules for the ultrafilter monad $\beta$ on $\Rel$.
 
Although \cref{kelley} has already been known to some degree (see \cite[Proposition 25]{guitart}), it seems that the limit sketch for topological spaces has never been written down explicitly and used before, in particular by means of defining the category $\Top(\C)$ of \emph{net-based topological space objects} internal to any complete category $\C$. We study categorical properties of this construction in \cref{sec:topspace}.

\begin{propintro}[\cref{topofunk}]
When $\C$ is wellpowered and complete, then the forgetful functor
\[\Top(\C) \to \C\]
is topological. Hence, if $\C$ is also cocomplete, $\Top(\C)$ is cocomplete as well.
\end{propintro}

By a formal dualization we also get the notion of a (net-based) \emph{cotopological space object} internal to a cocomplete category $\C$, and these form a category $\CoTop(\C)$.

\cref{mainintro,kelley} imply the following universal property of $\Top$ as a cocomplete category.
 
\begin{thmintro}[\cref{topue}] \label{topueintro}
For a cocomplete category $\C$ we have
\[\Hom_{\c}(\Top,\C) \simeq \CoTop(\C).\]
Hence, $\Top$ is the universal example of a cocomplete category with a cotopological space object.
\end{thmintro}

Here, a cotopological space object $X$ is mapped to a suitably defined tensor product functor $X \otimes_{\T} (-) : \Top \to \C$ that merges the cotopology with the topology. From \cref{topueintro} and a classification of cotopologies in $\Set$ and $\Top$ we then deduce the following equivalences, which were proven before by Isbell, see below.

\begin{thmintro}[\cref{cocontop,cocontop2}] \label{topmain}
There are equivalences of categories
\begin{align*}
\Hom_{\c}(\Top,\Set) & \simeq \Set, \\
\Hom_{\c}(\Top,\Top) & \simeq \{\text{topological topologies}\}.
\end{align*}
\end{thmintro}

Here, a \emph{topological topology} (as introduced by Isbell \cite{isbell3}) is a topological space equipped with a topology on its set of open subsets such that arbitrary unions and finite intersections become continuous operations on that set of open subsets.

Finally, in order to study \emph{continuous} functors on $\Top$, in \cref{sec:topop} we give a limit sketch that models the dual category $\Top^{\op}$. This is based on its description as an infinitary quasi-variety by Adámek, Barr and Pedicchio \cite{adamekpedicchio,barrpedicchio}. It also produces the category $\smash{\Topf(\C)}$ of \emph{frame-based topological space objects} internal to any complete category $\C$. These can be basically seen as internal locales \emph{with} points, keeping in mind that locales represent spaces \emph{without} points.

\cref{mainintro} yields the following result, where $\Hom^{\c}$ denotes the category of continuous functors.

\begin{thmintro}[\cref{contop}] \label{topaintro}
For a complete category $\C$ we have
\[\Hom^{\c}(\Top,\C) \simeq \Topf(\C)^{\op}.\]
Hence, up to this dualization, $\Top$ is the universal complete category with an internal frame-based topological space object.
\end{thmintro}

From this we will obtain yet another, even simpler proof of \cref{topmain}.

\textbf{Relation to Isbell's work.} While writing this paper, it became more and more clear that many of our results are already either explicit or implicit in Isbell's early work. Both equivalences in \cref{topmain} have already been proven in \cite[Theorems 1.3 and 1.4]{isbell3}, although with different methods. Isbell uses clever topological arguments and does not (explicitly) use the theory of limit sketches. We believe that our proof is more conceptual and easy to understand, in particular the one via \cref{topaintro}. In the introduction of \emph{loc.cit.}\ Isbell sketches several special cases of \cref{mainintro} and refers to his \cite[Theorem 3.8]{isbell2} for a general explanation, but that theorem seems to cover only the fully faithfulness part of \cref{mainintro} in a similar situation. However, \cite[Theorem 3.4]{isbell2} is actually a generalization of \cref{mainintro}. Also, \cref{compchar} is very similar to \cite[Theorem 2.7]{isbell2}. There are at least three technical differences: Isbell does not restrict to locally small categories, uses a theory notion that is more general than limit sketches (actually, three of them: prefacts, midfacts, facts), and he wants the category of models to have a forgetful functor to $\Set$. But most notably it is a challenge to even understand all the new terminology in Isbell's paper on functorial semantics. Even the translation from \cite[Theorem 3.4]{isbell2} to \cref{mainintro} requires some work. We therefore cannot claim any originality here, but hope to make Isbell's results more accessible. Moreover, we will explain several examples of them.
  
\textbf{Acknowledgments.} I would like to thank Ivan Di Liberti for providing many useful comments on the paper. I also would like to thank Zhen Lin Low, Tim Campion, Nathanael Arkor, Dmitri Pavlov, Eric Wofsey, Simon Henry, Nelson Martins-Ferreira and Chris Custer for helpful comments on the subject.

\textbf{Conventions.} Categories are assumed to be locally small unless otherwise stated, since we need to make this assumption for many constructions. We also assume (co)cones and (co)limits to be small unless otherwise stated, (co)completeness means the existence of small (co)limits, and likewise (co)continuity of a functor means that small (co)limits are preserved. Terminal objects are denoted by $1$, initial objects by $0$. We will use the language of generalized elements \cite[Section 2.3]{awodey}. We denote by $\Hom_{\c}$ categories of cocontinuous functors, and by $\Hom^{\c}$ categories of continuous functors (both of which are not necessarily locally small).


\tableofcontents


\section{Limit sketches and their models}\label{sec:limit}

This section is expository and recalls well-known definitions. Some texts focus on small sketches only and therefore just call them sketches \cite{adamekrosicky,barrwells,lair}. However, since general (possibly large) sketches are our focus, we will explicitly say when a sketch is small.
 
\begin{defi} \label{sketchdef}
Limit sketches and their models are defined as follows.
\begin{itemize}
\item A \emph{limit sketch} (or \emph{projective sketch}) $(\E,\S)$ consists of a category $\E$ and a class $\S$ of cones in $\E$. Sometimes we just write $\S$ for the limit sketch. The cones in $\S$ are called \emph{distinguished cones} of the sketch. We write cones simply as $\{A \to A_i\}_{i \in I}$, where $I$ is a small index category. We say that $(\E,\S)$ is \emph{small} if $\E$ is small and $\S$ is a set.
\item A \emph{model} of a limit sketch $(\E,\S)$ is a functor $M : \E \to \Set$ such that for every distinguished cone $\{A \to A_i\}_{i \in I}$ the induced cone $\{M(A) \to M(A_i)\}_{i \in I}$ is a limit cone, i.e.\ a universal cone, in the category $\Set$. Thus, models are required to map distinguished cones to limit cones.
\item More generally, if $\C$ is any category, a \emph{$\C$-valued model} of $(\E,\S)$ is a functor $M : \E \to \C$ that maps the distinguished cones in $\S$ to limit cones in $\C$. Here we usually assume that $\C$ is complete, or at least that limits indexed by the index categories of the cones in $\S$ exist. A morphism of models is defined as a morphism of the underlying functors. We obtain a (not necessarily locally small) category
\[\Mod_{\C}(\S) \subseteq \Hom(\E,\C)\]
of $\C$-valued models of $\C$, which we abbreviate by $\Mod(\S)$ in the standard case $\C=\Set$. When a category $\A$ is equivalent to $\Mod(\S)$, we say that $\S$ \emph{models} $\A$.
\item A morphism of limit sketches is a functor of the underlying categories that preserves distinguished cones. We get a (not necessarily locally small) category of limit sketches. Notice that a $\C$-valued model of a limit sketch $(\E,\S)$ is just a morphism of limit sketches $(\E,\S) \to (\C,\{\text{limit cones in } \C\})$.
\item A limit sketch $(\E,\S)$ is called \emph{realized} when for every object $A \in \E$ the representable functor $\Hom(A,-) : \E \to \Set$ is a model of $\S$, which we therefore call a \emph{representable model}. This means that all distinguished cones are limit cones in $\E$.
\item Dualization yields the notion of a \emph{colimit sketch} (or \emph{inductive sketch}). It consists of a category and a class of cocones in it, and a model of it is a functor that maps the distinguished cocones to colimit cocones. If $(\E,\S)$ is a limit sketch, then $(\E^{\op},\S^{\op})$ is a colimit sketch with the obvious definition of $\S^{\op}$, and for every category $\C$ we have
\[\Mod_{\C}(\S)^{\op} \cong \Mod_{\C^{\op}}(\S^{\op}).\]
\item A \emph{mixed sketch} (or just \emph{sketch}) consists of a category and classes of cones and cocones in it. We will not work with mixed sketches.
\end{itemize}
\end{defi}

\begin{rem}\label{compfree}
Some authors prefer a graph-based definition of a sketch, see for example \cite[Chapter 4]{barrwells} or \cite[Chapter 3]{makkai}, where our sketches are called category-sketches. In a graph-sketch the cones live in a (possibly large) graph and are indexed by graphs as well. The graph also has pairs of paths to model commutativity relations. The advantage of this approach is that it is straight forward to describe mathematical objects as models of a graph-sketch, and the word \emph{sketch} fits very well. We can always replace a graph-sketch by a category-sketch by taking the path category modulo the commutativity relations, and the models are the same. However, it is not so clear when this category is locally small (which we certainly want). For example, it is straight forward to give a locally small graph-sketch for complete boolean algebras, but the induced category-sketch is much harder to describe and in fact not locally small by \cite[Theorem 4]{hales}. Another advantage of category-sketches is the beauty and simplicity of the language of functors, and for realized limit category-sketches representable models are immediately available and very useful. Let us mention that Ehresmann's notion of a neocategory (multiplicative graph) unifies both notions of sketches \cite{ehresmann3}.
\end{rem}

\begin{rem} \label{internalize}
As soon as the objects of a category can be identified with models of a limit sketch $\S$, we can internalize their definition into arbitrary categories $\C$, namely as $\C$-valued models of $\S$. For instance, the limit sketch of a group \cite[Chapter 4, Section 1.3]{barrwells} yields the usual notion of a group object internal to a category.
\end{rem}

\begin{rem} \label{real}
Every \emph{small} limit sketch can be replaced by a small realized limit sketch with the same models \cite[Section I.3]{ehresmann3}. But this does not necessarily work with general (locally small) limit sketches.
\end{rem}

\begin{rem} \label{modyon}
If $(\E,\S)$ is a limit sketch, a functor $M : \E \to \C$ is a $\C$-valued model of $\S$ if and only if for every $T \in \C$ the functor $\Hom(T,M(-)) : \E \to \Set$ is a $\Set$-valued model of $\S$. Accordingly, for a colimit sketch $(\E,\S)$ a functor $M : \E \to \C$ is a $\C$-valued model if and only if for every $T \in \C$ the functor $\Hom(M(-),T) : \E^{\op} \to \Set$ is a $\Set$-valued model of $\S^{\op}$.
\end{rem}

\begin{lemma} \label{limitsofmodels}
Let $(\E,\S)$ be a (co)limit sketch. Let $\C$ be a (co)complete category. Then $\Mod_\C(\S)$ is (co)complete. In fact, (co)limits are computed pointwise, so that for each $A \in \E$ the evaluation functor $\Mod_\C(\S) \to \C$, $M \mapsto M(A)$ is (co)continuous.
\end{lemma}

\begin{proof}
Since (co)limits commute with (co)limits, $\Mod_{\C}(\S)$ is clearly closed under (co)limits in the (usually not locally small) category $\Hom(\E,\Set)$, which is (co)complete with pointwise (co)limits.
\end{proof}

\begin{cor} \label{nec}
A necessary condition for a category $\A$ to be modeled by a (co)limit sketch is the following: $\A$ is (co)complete, and the (co)continuous functors $\A \to \Set$ are jointly conservative (i.e.\ a morphism in $\A$ is an isomorphism when its image through any (co)continuous functor $\A \to \Set$ is an isomorphism).
\end{cor}

\begin{rem} \label{modfunk}
If $(\E,\S)$ is a (co)limit sketch, $M : \E \to \C$ is a $\C$-valued model of $\S$, and $H : \C \to \D$ is a (co)continuous functor, it is clear that $H \circ M$ is a $\D$-valued model of $\S$. Thus, $H$ induces a (co)continuous functor
\[H_* : \Mod_{\C}(\S) \to \Mod_{\D}(\S).\]
\end{rem}

\begin{defi} \label{tensorsketch}
If $(\E,\S)$ and $(\F,\T)$ are limit sketches, we define their tensor product by $(\E \times \F, \S \otimes \T)$, where $\S \otimes \T$ consists of those cones of the form $(A,\tau)$ or $(\sigma,B)$, where $A \in \E$, $B \in \F$, $\sigma \in \S$ and $\tau \in \T$. This extends to a symmetric monoidal structure on limit sketches. If $\C$ is a complete category, the natural equivalence $\Hom(\E \times \F,\C) \simeq \Hom(\E,\Hom(\F,\C))$ clearly restricts to an equivalence
\[\Mod_{\C}(\S \otimes \T) \simeq \Mod_{\Mod_{\C}(\T)}(\S).\]
\end{defi}

\begin{rem} \label{smallsketch}
If $\S$ is a small limit sketch, then clearly $\Mod(\S)$ is locally small, and by \cite[Theorem 4.2.1]{barrwells} or \cite[Theorem 1.39]{adamekrosicky} it is a reflective subcategory of $\Hom(\E,\Set)$. In particular, $\Mod(\S)$ is cocomplete. The reflector is based on a transfinite construction that may fail to work for large limit sketches. The construction in \cite{velebil} is not a solution to this problem since it just interprets the sketch as a small sketch with respect to a larger Grothendieck universe. So for large limit sketches $\S$ it may happen that $\Mod(\S)$ is not cocomplete, even when $\S$ is realized. In fact, an example for this failure has been given by Simon Henry at \url{https://mathoverflow.net/questions/393511} using ideas from algebraic set theory.
\end{rem}

Next we find a condition when $\Mod_{\C}(\S)$ is locally small.

\begin{defi} \label{closure}
Let $(\E,\S)$ be a limit sketch. We define the \emph{closure} $\overline{\S}$ as the class of cones in $\E$ that are mapped to limit cones in $\Set$ by any model of $\S$, and hence to limit cones in $\C$ by any $\C$-valued model by \cref{modyon}. Clearly, we have $\S \subseteq \overline{\S}$ and
\[\Mod_{\C}(\S) = \Mod_{\C}(\overline{\S}).\]
In practice, every cone that is recursively built from cones in $\S$ belongs to $\overline{\S}$.
\end{defi}

\begin{defi} \label{monocone}
Notice that a morphism $f : X \to Y$ in a category is a monomorphism if and only if 
\[\begin{tikzcd}
X \ar{r}{\id_X} \ar{d}[swap]{\id_X} & X \ar{d}{f} \\ X \ar{r}[swap]{f} & Y
\end{tikzcd}\]
is a pullback. More generally, a family $\{f_i: X \to Y_i : i \in I\}$ is jointly monomorphic if and only if the cone
\[\begin{tikzcd}
X \ar{r}{\id_X} \ar{d}[swap]{\id_X} & X \ar{d}{f_i} \\ X \ar{r}[swap]{f_i} & Y_i
\end{tikzcd}\]
whose legs vary over all $i \in I$ is a limit cone. We therefore call this the \emph{monomorphism cone} of the family $\{f_i : X \to Y_i : i \in I\}$.
\end{defi}

\begin{lemma} \label{locallysmall}
Let $(\E,\S)$ be a limit sketch and $\C$ be a category. Assume that $\E$ has a small full subcategory $\E' \subseteq \E$ such that for every $A \in \E$ the monomorphism cone of the family $\{f : A \to B : B \in \E', f : A \to B\}$ belongs to the closure $\overline{\S}$. Then $\Mod_{\C}(\S)$ is locally small.
\end{lemma}

\begin{proof}
The assumption implies that for every $\C$-valued $\S$-model $M$ and every object $A \in \E$ the family $\{M(A) \to M(B) : B \in \E',\, A \to B\}$ is jointly monomorphic in $\C$. Thus, if $N : \E \to \C$ is any functor, morphisms $N \to M$ are determined by their restrictions to $\E'$.
\end{proof}


\section{Compactness} \label{sec:compact}

Isbell \cite{isbell1} calls a category $\A$ \emph{compact} when every hypercocontinuous functor on $\A$ is a left adjoint, where a functor is called \emph{hypercocontinuous} if it preserves all (possibly large) colimits. We refer to \cite{borger,ulmer} for results on compact categories. Since our colimits are small by definition, we need a slightly different notion, which is of course not new, only the name is.

\begin{defi}
We call a category $\A$ \emph{strongly compact} when every cocontinuous functor $\A \to \B$ into some category $\B$ is a left adjoint.
\end{defi}
 
Durov calls such categories \emph{complete} or \emph{total} in the context of his theory of vectoids \cite{durov}; both terms however already have a different meaning and therefore should be avoided. Rattray has abbreviated the dual condition by \emph{LAP} that stands for ``left adjoint property'' \cite{rattray}. In this terminology, strongly compact categories have \emph{RAP}, the ``right adjoint property''. We prefer to describe these categories with an adjective instead.

The following result is already known for compact categories and can be proven in the same way.

\begin{lemma} \label{isbellchar}
The following are equivalent for a category $\A$.
\begin{enumerate}
\item $\A$ is strongly compact.
\item Every cocontinuous functor $\A \to \B$ into a cocomplete category $\B$ is a left adjoint.
\item Every continuous functor $\A^{\op} \to \Set$ is representable.
\item The corestricted Yoneda embedding $\A \hookrightarrow \Hom^{\c}(\A^{\op},\Set)$ is an equivalence.
\end{enumerate}
\end{lemma}

\begin{proof}
$(1) \Rightarrow (2)$ and $(3) \Leftrightarrow (4)$ are trivial. $(2) \Rightarrow (3)$. If $F : \A^{\op} \to \Set$ is continuous, then $F^{\op} : \A \to \Set^{\op}$ is cocontinuous, hence it has a right adjoint $G : \Set^{\op} \to \A$. Then $F$ is represented by $G(1^{\op}) \in \A$ because of
\[F(A) \cong \Hom(1,F(A)) \cong \Hom(F^{\op}(A),1^{\op}) \cong \Hom(A,G(1^{\op})).\]
$(3) \Rightarrow (1)$. Let $H : \A \to \B$ be a cocontinuous functor. Then for every $B \in \B$ the functor $\Hom(H(-),B) : \A^{\op} \to \Set$ is continuous and therefore representable, which precisely means that $H$ is a left adjoint.
\end{proof}

We remark that the corestricted Yoneda embedding $\A \hookrightarrow \Hom^{\c}(\A^{\op},\Set)$ has been studied in general by Lambek \cite{lambek} and Makkai-Paré \cite[§6.4]{makkai}.
 
\begin{cor} \label{isbelllimit}
Every strongly compact category is the category of models of a (possibly large) realized limit sketch, in particular it is complete. The sketch can be chosen so that every model is representable.
\end{cor}

\begin{proof}
If $\A$ is strongly compact, then \cref{isbellchar} implies that $\A$ is modeled by the limit sketch $(\A^{\op},\{\text{all limit cones}\})$.
\end{proof}

In \cref{modiscompact} below we will prove a (stronger) converse of \cref{isbelllimit}.

\begin{rem} \label{adjointgame}
If $\A$ and $\B^{\op}$ are strongly compact, then taking the right or left adjoint implements an equivalence of categories
\[\Hom_{\c}(\A,\B) \simeq \Hom^{\c}(\B,\A)^{\op}.\]
This makes strongly compact categories useful: we can easily switch between cocontinuous and continuous functors.
\end{rem}

\begin{rem} \label{saft}
The Special Adjoint Functor Theorem precisely says that any cocomplete wellcopowered category with a generator is strongly compact. This gives lots of examples of strongly compact categories -- every locally presentable category, but also $\Top$, $\Set^{\op}$ and $\Ab^{\op}$ for instance. However, $\Grp^{\op}$ is not compact \cite{isbell1} and hence not strongly compact: for every infinite cardinal $\lambda$ we pick a simple group $G_{\lambda}$ with $\lambda$ elements (for example, the alternating group on $\lambda$ elements). The functor
\[\textstyle \prod_{\lambda} \Hom(G_{\lambda},-) : \Grp \to \Set\]
is well-defined, continuous and not representable.
\end{rem}

\begin{rem} \label{rattray}
Rattray has proven that every category that is monadic over a strongly compact category is also strongly compact \cite[Theorem 1]{rattray}. The same proof works for compact categories \cite[Remarks 2.3]{borger}.
\end{rem}
 
\begin{rem}\label{topiscomp}
By \cref{saft} and \cref{isbelllimit} we see that $\Top$ is the category of models of a realized limit sketch, which has to be large since $\Top$ is not locally presentable. However, the limit sketch from the proof is \emph{very} large -- it contains \emph{every} space -- and therefore not very useful in practice. Specifically, it does not give a very concrete way of internalizing topological spaces: Following \cref{internalize} we could define topological space objects internal to a complete category $\C$ as continuous functors
\[\Top^{\op} \to \C,\]
but how do they look like, really? Using strong compactness of $\Top$, it is not hard to see that they correspond to continuous functors
\[\C^{\op} \to \Top\]
whose composition $\C^{\op} \to \Top \to \Set$ is representable. But what does this mean, specifically, in terms of data inside of $\C$? We will answer these questions in \cref{sec:topspace,sec:topsketch}.
\end{rem}

\begin{rem}
The category $\Meas$ of measurable spaces is also strongly compact by \cref{saft} and hence modeled by a large realized limit sketch. It is an interesting question if we can shrink that sketch down to something more useful (as we will do for $\Top$).
\end{rem}

\begin{rem}
Not every strongly compact category is cocomplete. For example, Adámek has constructed a monad $\T$ on the category $\Gra$ of simple directed graphs such that $\Alg(\T)$ is not cocomplete \cite{adamek}. Since $\Gra$ is locally finitely presentable and hence strongly compact, $\Alg(\T)$ is strongly compact as well by \cref{rattray}. On the other hand, one can argue that cocompleteness should be added to the definition of a strongly compact category because we are dealing with cocontinuous functors on them. Cocomplete strongly compact categories have been studied by Ivan Di Liberti under the name \emph{LAFT categories} (unpublished work). The following result also supports this adjustment.
\end{rem}

\begin{lemma} \label{compactfun}
If $\A$ is cocomplete and strongly compact, then for every small category $\I$ the functor category $\Hom(\I,\A)$ is strongly compact as well.
\end{lemma}

\begin{proof}
This follows from \cite[Theorem 6.2]{andre}. Since that paper uses old terminology, we give a proof for the convenience of the reader. Let $H : \Hom(\I,\A) \to \B$ be a cocontinuous functor, where $\B$ is cocomplete. It corresponds to the functor $G : \I^{\op} \to \Hom_{\c}(\A,\B)$ defined by (where $\otimes$ denotes the copower)
\[G(i)(A) \coloneqq H\bigl(\Hom(i,-) \otimes A\bigr),\]
and conversely we have to coend expression
\[H(F) = \int^{i \in \I} G(i)\bigl(F(i)\bigr)\]
for $F \in \Hom(\I,\A)$ because of the fundamental formula (sometimes called the co-Yoneda Lemma)
\[F = \int^{i \in \I} \Hom(i,-) \otimes F(i).\]
In fact, this describes an equivalence (see also \cite[Lemma 1.1]{pitts} or \cite[Section 4]{andre})
\[\Hom_{\c}\bigl(\Hom(\I,\A),\B\bigr) \simeq \Hom\bigl(\I^{\op},\Hom_{\c}(\A,\B)\bigr).\]
Since $\A$ is strongly compact, for each $i \in \I$ there is a right adjoint $R(i) : \B \to \A$ of $G(i)$, and every morphism $i \to j$ in $\I$ induces a morphism $G(j) \to G(i)$, hence a morphism $R(i) \to R(j)$. Thus, $R$ is a functor $\I \to \Hom(\B,\A)$, which yields a functor $R' : \B \to \Hom(\I,\A)$.

We claim that $R'$ is right adjoint to $H$. To see this, let $F \in \Hom(\I,\A)$ and $B \in \B$. Then
\begin{align*}
\Hom(H(F),B) & \cong \int_{i \in \I} \Hom\bigl(G(i)(F(i)),B\bigr) \\
& \cong \int_{i \in \I} \Hom\bigl(F(i),R(i)(B)\bigr) \\
& \cong \int_{i \in \I} \Hom\bigl(F(i),R'(B)(i)\bigr) \\
& \cong \Hom(F,R'(B)). \qedhere
\end{align*}
\end{proof}
 
A lot more things could be said (and asked) about strongly compact categories, but we will not digress any further here.


\section{Small models} \label{sec:small}

The category of all $\S$-models of a limit sketch $\S$ is sometimes too large. For example, models of the empty limit sketch $(\E,\emptyset)$ are just functors $\E \to \Set$, which shows that $\Mod(\S)$ does not have to be locally small and therefore can be regarded as pathological. We need to look at a locally small subcategory of \emph{small models}. First, we recall the notion of a small functor:

\begin{defi}
A functor $\E \to \Set$ is called \emph{small} when it is a colimit of representable functors; recall that colimits are assumed to be small. The category of small functors $\E \to \Set$ is denoted by $\Hom^{\s}(\E,\Set)$ .
\end{defi}

\begin{prop} \label{small}
Let $M : \E \to \Set$ be a functor. The following conditions are equivalent:
\begin{enumerate}
\item The functor $M$ is small.
\item There is a functor $M': \E' \to \Set$ on a small category $\E'$ such that $M$ is the left Kan extension of $M'$ along a functor $\E' \to \E$.
\item There is a small full subcategory $\E' \subseteq \E$ such that $M$ is the left Kan extension of its restriction $M|_{\E'} : \E' \to \Set$ along the inclusion $\E' \hookrightarrow \E$.
\end{enumerate}
Moreover, the category $\Hom^{\s}(\E,\Set)$ is locally small and closed under colimits in $\Hom(\E,\Set)$. In particular, it is cocomplete.
\end{prop}

\begin{proof}
The equivalence between the conditions is proven in \cite[Proposition 4.83]{kelly}. Notice that Kelly uses the term \emph{accessible functor}, which nowadays has a different meaning \cite[Definition 2.16]{adamekrosicky}. The third condition can be used to show that $\Hom^{\s}(\E,\Set)$ is closed under colimits in $\Hom(\E,\Set)$, see \cite[Proposition 5.34]{kelly}. If $M : \E \to \Set$ is a small functor, say
\[M \cong {\colim}_{i \in I} \Hom(A_i,-),\]
and $N : \E \to \Set$ is any functor, then
\[\Hom(M,N) \cong {\lim}_{i \in I} \Hom(\Hom(A_i,-),N) \cong {\lim}_{i \in I} N(A_i)\]
is a set, showing that $\Hom^{\s}(\E,\Set)$ is locally small.
\end{proof}

Actually, $\Hom^{\s}(\E,\Set)$ is the free cocompletion of $\E^{\op}$ by \cite[Proposition 5.35]{kelly}, but our theory will contain this result as a special case (\cref{kov}).

\begin{defi} \label{smalldef}
Let $(\E,\S)$ be a realized limit sketch. A model $M : \E \to \Set$ of $\S$ is called \emph{small} when $M$ is a colimit of representable models inside $\Mod(\S)$. We denote the category of small models by $\Mod^{\s}(\S)$.
\end{defi}

\begin{rem}
For $S=\emptyset$ we recover the notion of a small functor. Notice that we do not require that the colimit takes place in $\Hom(\E,\Set)$, which would be a much too strong condition. Thus, the underlying functor of a small model does not have to be small. Notice that $\Mod^{\s}(\S)$ is indeed locally small by the same argument as in \cref{small}. Unfortunately, it is not cocomplete in this generality.
\end{rem}

\begin{rem}
More generally, we may call a model $M$ \emph{weakly small} if $M$ is an iterated colimit of representable models, which means that it belongs to the smallest class of models closed under colimits and containing the representable models. Even though our theory also works for weakly small models, we will not use that notion, also because in all our examples of interest except for $\S = \emptyset$ it will turn out that every model is already small.
\end{rem}

\begin{lemma} \label{smallchar}
Let $(\E,\S)$ be a realized limit sketch. Let $M : \E \to \Set$ be an $\S$-model such that there is small full subcategory $\E' \subseteq \E$ with the property that for every $\S$-model $N$ the restriction map
$\Hom(M,N) \to \Hom(M|_{\E'},N|_{\E'})$
is an isomorphism. Then $M$ is a small $\S$-model.
\end{lemma}

\begin{proof}
Consider the small category $I$ of elements of $M|_{\E'}$ that has objects $(A',m)$ with $A' \in \E'$ and $m \in M(A')$, and morphisms are defined in the obvious way. Then for every $\S$-model $N$ we have
\[\lim_{(A',m) \in I} \Hom(\Hom(A',-),N)  \cong \lim_{(A',m) \in I} N(A') \cong \Hom(M|_{\E'},N|_{\E'}) \cong \Hom(M,N).\]
Thus, in $\Mod(S)$ we have an isomorphism
\[M \cong \colim_{(A',m) \in I} \Hom(A',-). \qedhere\]
\end{proof}

\begin{rem}
We do not know if the converse of \cref{smallchar} is true (which is certainly the case for $\S=\emptyset$). The natural proof would start with a colimit expression $M \cong \colim_{i \in I} \Hom(A_i,-)$ induced by elements $u_i \in M(A_i)$ and defining $ \E' \coloneqq \{A_i : i \in I\}$. Then the restriction map $\Hom(M,N) \to \Hom(M|_{\E'},N|_{\E'})$ is injective, and every $\beta \in \Hom(M|_{\E'},N|_{\E'})$ induces some $\alpha \in \Hom(M,N)$ with $\alpha_{A_i}(u_i) = \beta_{A_i}(u_i)$, but it is not clear if this implies $\alpha_{A_i} = \beta_{A_i}$.
\end{rem}


\section{Universal property of the category of models} \label{sec:univ}

Let us recall the notion of the tensor product of functors, which is in fact a special case of a weighted colimit \cite[Section 3.1]{kelly}.

\begin{defi} \label{tp}
Let $\E,\C$ be categories. Let $M : \E \to \Set$ and $N : \E^{\op} \to \C$ be two functors. A \emph{tensor product} $N \otimes_{\E} M$ is an object of $ \C$ together with natural isomorphisms
\[\Hom(N \otimes_{\E} M,T) \cong \Hom\bigl(M,\Hom(N(-),T)\bigr)\]
for $T \in \C$. Here, we have $\Hom(N(-),T) : \E \to \Set$.
\end{defi}

\begin{prop} \label{tpex}
Let $(\E,\S)$ be a realized limit sketch. Let $M : \E \to \Set$ be a small model of $\S$, and let $N : \E^{\op} \to \C$ be a $\C$-valued model of $\S^{\op}$, where $\C$ is cocomplete. Then the tensor product $N \otimes_{\S} M \coloneqq N \otimes_{\E} M$ exists. Moreover, it defines a functor
\[\Mod^{\s}(\S) \to \C,\quad M \mapsto N \otimes_{\S} M\]
with natural isomorphisms
\[N \otimes_{\S} \Hom(A,-) \cong N(A).\]
\end{prop}

\begin{proof}
The proof uses standard techniques from category theory, but we spell out the details for the convenience of the reader. We need to show that for every small model $M$ the functor
\[\C \to \Set, ~ T \mapsto \Hom\bigl(M,\Hom(N(-),T)\bigr)\]
is representable. Notice that $\Hom(N(-),T)$ is a model of $\S$ precisely since $N$ is a model of $\S^{\op}$ (\cref{modyon}). Since $M$ is small, there is a colimit decomposition $M = \colim_i M_i$ with representable models $M_i$. It follows
\[\Hom\bigl(M,\Hom(N(-),T)\bigr) \cong {\lim}_i \Hom\bigl(M_i,\Hom(N(-),T)\bigr).\]
Since $\C$ is cocomplete, a limit of representable functors $\C \to \Set$ is representable (namely, by the colimit of the representing objects). Therefore, it suffices to show the claim when $M$ is representable, say $M = \Hom(A,-)$ for some $A \in \E$. In that case, the Yoneda Lemma implies that the functor is isomorphic to $\Hom(N(A),-)$, thus represented by $N(A) \in \C$, and we are done. We can make this a bit more explicit: If $M = \colim_i \Hom(A_i,-)$ is a colimit of representable models, then $\colim_i N(A_i)$ satisfies the universal property of $N \otimes_{\S} M$. Any morphism $M \to M'$ induces a natural morphism $\Hom\bigl(M',\Hom(N(-),T)\bigr) \to \Hom\bigl(M,\Hom(N(-),T)\bigr)$ and hence a morphism $N \otimes_{\S} M \to N \otimes_{\S} M'$ by the Yoneda Lemma.
\end{proof}

\begin{prop} \label{tpco}
With the notation of \cref{tpex}, for every $\C$-valued $\S^{\op}$-model $N$ the functor $- \otimes_{\S} N : \Mod^{\s}(\S) \to \C$ preserves all colimits of $\Mod^{\s}(\S)$ that are preserved by the inclusion $\Mod^{\s}(\S) \hookrightarrow \Mod(\S)$.
\end{prop}

\begin{proof}
Even though $- \otimes_{\S} N$ does not have to be a left adjoint (since the models $\Hom(N(-),T)$ are not small in general), we can just recycle the well-known argument that left adjoints preserve colimits. Let $M = \colim_i M_i$ be a colimit in $\Mod^{\s}(\S)$ that is preserved by the inclusion $\Mod^{\s}(\S) \hookrightarrow \Mod(\S)$. In particular, we may apply its universal property to models of the form $\Hom(N(-),T)$, where $ T \in \C$. Therefore,
\begin{align*}
\Hom(M \otimes_{\S} N,T) & \cong \Hom\bigl(M,\Hom(N(-),T)\bigr) \\
& \cong {\lim}_i \Hom\bigl(M_i,\Hom(N(-),T)\bigr) \\
& \cong {\lim}_i \Hom(M_i \otimes_{\S} N,T),
\end{align*}
which shows $M \otimes_{\S} N \cong {\colim}_i M_i \otimes_{\S} N$.
\end{proof}

This motivates the following definition.

\begin{defi}
A functor $\Mod^{\s}(\S) \to \C$ is called \emph{weakly cocontinuous} if it preserves all colimits that are preserved by the inclusion $\Mod^{\s}(\S) \hookrightarrow \Mod(\S)$. These form a (not necessarily locally small) category $\Hom_{\wc}(\Mod^{\s}(\S),\C)$.
\end{defi}
 
We now describe the universal $\S^{\op}$-model.

\begin{prop}\label{comod}
Let $(\E,\S)$ be a realized limit sketch. Then we may corestrict the Yoneda embedding $\E^{\op} \to \Hom(\E,\Set)$, $A \mapsto \Hom(A,-)$ to a functor
\[Y : \E^{\op} \to \Mod^{\s}(\S),\]
which is, in fact, a model of $\S^{\op}$ with values in $\Mod^{\s}(\S)$, and likewise for $\Mod(\S)$. We call it the \emph{Yoneda model}.
\end{prop}

\begin{proof}
Since $(\E,\S)$ is realized, $\Hom(A,-)$ is a model of $\S$, and it is small for trivial reasons, so that we get a functor $Y : \E^{\op} \to \Mod^{\s}(\S)$. Now, let $\{A \to A_i\}$ be a distinguished cone. We need to show that $\{Y(A_i) \to Y(A)\}$ is a colimit cocone in $\Mod^{\s}(\S)$. In fact, it is even a colimit cocone in $\Mod(\S)$. For this we need to check that for every $\S$-model $M$ the cone $\{\Hom(Y(A),M) \to \Hom(Y(A_i),M)\}$ is a limit cone in $\Set$. But the Yoneda Lemma identifies this cone with $\{M(A) \to M(A_i)\}$, which is a limit cone by the very definition of a model.
\end{proof}

We can now prove our main theorem.

\begin{thm}\label{maintechnical}
Let $(\E,\S)$ be a realized limit sketch. Then, for any cocomplete category $\C$ there is a natural equivalence
\begin{align*}
\Hom_{\wc}(\Mod^{\s}(\S),\C) & ~\simeq~ \Mod_{\C}(\S^{\op}),\\
F & ~\mapsto~ F \circ Y, \\
N \otimes_{\S} - & ~\mapsfrom~ N.
\end{align*}
\end{thm}
 
\begin{proof}
Let $F : \Mod^{\s}(\S) \to \C$ be a weakly cocontinuous functor. By \cref{comod} the Yoneda model $Y : \E^{\op} \to \Mod^{\s}(\S)$ maps all distinguished cones to colimit cocones that are preserved by $\Mod^{\s}(\S) \hookrightarrow \Mod(\S)$. Thus, $F \circ Y : \E^{\op} \to \C$ is a $\C$-valued $\S^{\op}$-model, and it is clear that any morphism $F \to F'$ induces a morphism $F \circ Y \to F' \circ Y$. Conversely, if $N : \E^{\op} \to \C$ is a $\C$-valued $\S^{\op}$-model, then by \cref{tpco} the functor $N \otimes_{\S} - : \Mod^{\s}(\S) \to \C$ is weakly cocontinuous, and the universal property of tensor products shows that every morphism $N \to N'$ induces a morphism $N \otimes_{\S} - \to N' \otimes_{\S} -$. The natural isomorphisms $N \otimes_{\S} Y(A) \cong N(A)$ from \cref{tpex} show $(N \otimes_{\S} -) \circ Y \cong N$ for any $\S^{\op}$-model $N$. Conversely, if $F : \Mod^{\s}(\S) \to \C$ is a weakly cocontinuous functor, we will construct a morphism
\[\alpha: (F \circ Y) \otimes_{\S} - \to F\]
as follows: For $ M \in \Mod^{\s}(\S)$ the morphism $\alpha_M : (F \circ Y) \otimes_{\S} M \to F(M)$ corresponds to the morphism $M \to \Hom\bigl(F(Y(-)),F(M)\bigr)$ of models whose $A$-component (for $A \in \E$) is defined by
\[M(A) \cong \Hom(Y(A),M) \xrightarrow{F} \Hom\bigl(F(Y(A)),F(M)\bigr).\]
If $M = Y(A)$ is representable, then $\alpha_M$ is an isomorphism since it is equal to the isomorphism $(F \circ Y) \otimes_{\S} Y(A) \cong F(Y(A))$ from \cref{tpex}. If $M$ is an arbitrary small model, then $M$ is a colimit of representable models in $\Mod^{\s}(\S)$ and $\Mod(\S)$. Since both sides of $\alpha$ are weakly cocontinuous and thus preserve this colimit, it follows that $\alpha_M$ is an isomorphism as well.
\end{proof}

We will mostly use the following less technical special case, which actually follows from Isbell's \cite[Theorem 3.4]{isbell2} as mentioned in the introduction.

\begin{thm}\label{main}
Let $(\E,\S)$ be a realized limit sketch. Assume that every $\S$-model is small. Then, for any cocomplete category $\C$ there is a natural equivalence
\begin{align*}
\Hom_{\c}(\Mod(\S),\C) & ~\simeq~ \Mod_{\C}(\S^{\op}),\\
F & ~\mapsto~ F \circ Y, \\
N \otimes_{\S} - & ~\mapsfrom~ N.
\end{align*}
Thus, if $\Mod(\S)$ is cocomplete, $\Mod(\S)$ is the universal example of a cocomplete category with a model of $\S^{\op}$, namely the Yoneda model $Y : \E^{\op} \to \Mod(\S)$.
\end{thm}

\begin{proof}
The assumption $\Mod(\S)=\Mod^{\s}(\S)$ implies that weakly cocontinuous functors coincide with cocontinuous functors. So we are done by \cref{maintechnical}.
\end{proof}

\begin{cor} \label{modiscompact}
A category is strongly compact if and only if it is modeled by a realized limit sketch such that every model is small.
\end{cor}

\begin{proof}
The direction $\Rightarrow$ follows from \cref{isbelllimit}. To prove $\Leftarrow$ let $\S$ be a realized limit sketch such that every $\S$-model is small. By \cref{isbellchar} it is enough to prove that each cocontinuous functor $\Mod(\S) \to \C$ into a cocomplete category $\C$ is a left adjoint. By \cref{main} it has the form $N \otimes_{\S} -$ for some $N \in \Mod_{\C}(\S^{\op})$, which is left adjoint to $T \mapsto \Hom(N(-),T)$.
\end{proof}

\begin{cor} \label{kov}
Let $\E$ be a category. Then for any cocomplete category $\C$ there is a natural equivalence $\Hom_{\c}(\Hom^{\s}(\E,\Set),\C) \simeq \Hom(\E^{\op},\C)$. Thus, $\Hom^{\s}(\E,\Set)$ is the universal cocompletion of $\E^{\op}$.
\end{cor}

\begin{proof}
This is just the special case $\S=\emptyset$ of \cref{maintechnical}, using \cref{small}.
\end{proof}

We can also deduce the following generalized Eilenberg-Watts Theorem.

\begin{cor} \label{eilenbergmain}
Let $\S$ be a realized limit sketch such every $\S$-model is small. Let $\T$ be another limit sketch such that $\Mod(\T)$ is locally small and cocomplete. Then we have
\[\Hom_{\c}(\Mod(\S),\Mod(\T)) \simeq \Mod(\S^{\op} \otimes \T).\]
\end{cor}

\begin{proof}
This follows from \cref{tensorsketch} and \cref{main}.
\end{proof}

\begin{cor} \label{lfpue}
Let $\S$ be a small realized limit sketch. Then $\Mod(\S)$ is cocomplete, and for any cocomplete category $\C$ there is a natural equivalence
\[\Hom_{\c}(\Mod(\S),\C) \simeq \Mod_{\C}(\S^{\op}).\]
\end{cor}

\begin{proof}
Since $\S$ is small, we have $\Mod(\S)=\Mod^{\s}(\S)$, and this category is cocomplete by \cref{smallsketch}. Now use \cref{main}.
\end{proof}

\begin{rem}
Actually, in the small case one does not need the assumption that the limit sketch $\S$ is realized. This is because one can compose the Yoneda embedding with the reflector $\Hom(\E,\Set) \to \Mod(\S)$ from \cref{smallsketch}. Alternatively, we may argue with \cref{real}. \cref{lfpue} was first proven by Pultr \cite[Theorem 2.5]{pultr} (where small realized limit sketches are called relational theories). We refer to \cite[Theorem 2.2.4]{chirvasitu} for a more modern and readable exposition. What makes \cref{lfpue} remarkable is that it gives a ($2$-categorical) universal property of each locally presentable category, since every such category is modeled by a small realized limit sketch $\S$ \cite[Corollary 1.52]{adamekrosicky}. It is then the universal example of a cocomplete category with an internal model of $\S^{\op}$, which can also be seen as an internal \emph{comodel of $\S$}. The tricky part is to find a sketch that is sufficiently small.
\end{rem}

We now specialize our theorem to infinitary Lawvere theories, whose basic theory we will recall in \cref{sec:lawapp}; see \cref{lawdef} in particular.

\begin{prop}\label{lawsmall}
Let $\L$ be an infinitary Lawvere theory. Then every $\L$-model is small.
\end{prop}

\begin{proof}
By \cref{ismonadic} there is an equivalence $\Mod(\L) \simeq \Mod(\T)$ for a monad $\T$ on $\Set$ that maps representable $\L$-models onto free $\T$-algebras. Now we can use the well-known fact that every $\T$-algebra is a (reflexive) coequalizer of a pair of homomorphisms between free $\T$-algebras \cite[Proposition 3.3.7]{barrwells}.
\end{proof}

\begin{thm} \label{lawue}
Let $\L$ be an infinitary Lawvere theory. Then $\Mod(\L)$ is cocomplete, and for every cocomplete category $\C$ we have
\[\Hom_{\c}(\Mod(\L),\C) \simeq \Mod_{\C}(\L^{\op}).\]
\end{thm}

\begin{proof}
A classical result by Linton \cite{linton3} says that every monadic category over $\Set$ is cocomplete, which proves the first statement via \cref{ismonadic}. The rest follows from \cref{lawsmall} and \cref{main}.
\end{proof}

\begin{rem}
\cref{lawue} can be applied to finitary Lawvere theories as well because of \cref{finitary}. This special case is well-known and appears for example in \cite[Theorem 13]{poinsot}. The case that $\C$ is also modeled by a finitary Lawvere theory goes back to Freyd \cite{freyd}.
\end{rem}

\begin{defi}\label{internalalgebra}
Let $\T=(T,\eta,\mu)$ be a monad on $\Set$. A $\T$-algebra structure on a set $X$ is a map $T(X) \to X$ satisfying two axioms. It is easy to check \cite[Proposition 3.14]{manes2} that this is equivalent to a morphism of monads $\T \to \Hom(X^{-},X)$, where $\Hom(X^{-},X)$ is the double dualization monad of $X$ mapping a set $I$ to $\Hom(X^I,X)$. We can thus define $\T$-algebras internal to any category $\C$ with products and obtain a category $\Alg_{\C}(\T)$. Another way of defining $\Alg_{\C}(\T)$ is by combining the equivalence between monads and infinitary Lawvere (see \cref{sec:lawapp}) and using \cref{internalize}. Specifically, a $\T$-algebra internal to $\C$ is a product-preserving functor $\K^{\op} \to \C$, where $\K$ is the Kleisli category of $\T$. The two notions of internal $\T$-algebra agree by \cref{ismonadic}. By dualization, for every category with coproducts $\C$ we have a category of internal $\T$-\emph{coalgebras}
\[\CoAlg_{\C}(\T) \coloneqq \Alg_{\C^{\op}}(\T)^{\op},\]
where a $\T$-coalgebra structure on $X \in \C$ is a morphism of monads $\T \to \Hom(X,(-) \otimes X)$. 
\end{defi}

\begin{thm} \label{monue}
Let $\T$ be a monad on $\Set$. Then $\Alg(\T)$ is cocomplete, and for any cocomplete category $\C$ we have an equivalence
\[\Hom_{\c}(\Alg(\T),\C) \simeq \CoAlg_{\C}(\T).\]
\end{thm}

\begin{proof}
This is equivalent to \cref{lawue} because of the equivalences in \cref{ismonadic} and \cref{lawmon}.
\end{proof}

\begin{rem}
The variant of \cref{monue} with monads over an arbitrary cocomplete category preserving reflexive coequalizers appears as \cite[Proposition 6.2.15]{bra2}. A tensor categorical version of \cref{monue} appears as \cite[Proposition 4.2.3]{bra1}.
\end{rem}


\section{Examples} \label{sec:examples}

This section showcases several examples of our theorems in \cref{sec:univ}.

\begin{ex} \label{grpue}
For the category of groups \cref{monue} says that there is an equivalence of categories
\[\Hom_{\c}(\Grp,\C) \simeq \CoGrp(\C)\]
for every cocomplete category $\C$, where a cogroup structure on $X \in \C$ consists of three morphisms $X \to X \sqcup X$ (comultiplication), $X \to 0$ (counit), $X \to X$ (coinversion) satisfying the dualized group axioms. It follows that $\Grp$ is the universal example of a cocomplete category with a cogroup object, namely $\IZ$ with the cogroup structure induced by $\Hom(\IZ,G) \cong G$. For example, $\CoGrp(\Set)$ is trivial, and a classical result by Kan \cite{kan} implies $\CoGrp(\Grp) \cong \Set_*$. Many more examples of this kind can be found in Bergman's book \cite[Chapter 10]{bergman}.
\end{ex}

\begin{ex}\label{watts}
If $\C$ is a preadditive category, every object of $\C$ carries a unique cocommutative cogroup structure. Assume that $\C$ is also cocomplete. If $R$ is a ring and ${}_R \Mod$ (resp.\ $\Mod_R$) is the category of left (resp.\ right) $R$-modules, then \cref{monue} implies
\[\Hom_{\c}({}_R \Mod,\C) \simeq \Mod_{R}(\C).\]
The universal right $R$-module object in ${}_R \Mod$ is the $R$-bimodule $R$. The classical Eilenberg-Watts Theorem $\Hom_{\c}({}_R \Mod,{}_S \Mod) \simeq {}_S \Mod_R$ is the special case $\C={}_S \Mod$ for a ring $S$.
\end{ex}

\begin{ex} \label{sheaves}
Let $X$ be a topological space. Then the category $\Sh(X)$ of sheaves (of sets) on $X$ identifies with $\Mod(\S^{\op})$ for a suitable colimit sketch $\S$ in the category $\O(X)$ of open subsets of $X$. Therefore, \cref{lfpue} implies
\[\Hom_{\c}(\Sh(X),\C) \simeq \CoSh_{\C}(X),\]
where a $\C$-valued cosheaf on $X$ is a functor $F : \O(X) \to \C$ such that for every open covering $U = \bigcup_{i \in I} U_i$ the diagram $\coprod_{i,j \in I} F(U_i \cap U_j) \rightrightarrows \coprod_{i \in I} F(U_i) \to F(U)$ is exact. The universal cosheaf in $\Sh(X)$ is $U \mapsto \Hom_{\O(X)}(-,U)$.
\end{ex}
 
\begin{ex}\label{sheafmodules}
Let $(X,\O_X)$ be a ringed space and $\Mod(X,\O_X)$ be its category of $\O_X$-modules. We can model it with a limit sketch by gluing the limit sketches for $\O_X(U)$-modules. More precisely, for every open subset $U \subseteq X$ let $\L_U$ be the Lawvere theory of $\O_X(U)$-modules with generic object $T_U$. Every inclusion $V \subseteq U$ induces a ring homomorphism $\O_X(U) \to \O_X(V)$ and hence a morphism of Lawvere theories $\rho_{V,U} : \L_U \to \L_V$ (in particular, $\rho_{V,U}(T_U) \cong T_V$). Consider the Grothendieck construction
\[\E \coloneqq \int_{U \subseteq X \text{ open}} \L_U.\]
Objects of $\E$ are pairs $(U,A)$, where $U$ is a open subset of $X$ and $A \in \L_U$ is an object, and a morphism $(U,A) \to (V,B)$ is an inclusion $V \subseteq U$ and a morphism $\rho_{V,U}(A) \to B$ in $\L_V$. We have fully faithful functors $\L_U \hookrightarrow \E$. Let $\S$ consist of the images of all product cones in each $\L_U$ as well as of the following limit cone for every open covering $U = \bigcup_{i \in I} U_i$.
\[\begin{tikzcd}
& (U_i,T_{U_i}) \ar{dr} &  \\
(U,T_U) \ar{ur} \ar{dr} && (U_i \cap U_j, T_{U_i \cap T_j}) \\ 
& (U_j,T_{U_j}) \ar{ur} &
\end{tikzcd}\]
Here, the legs vary over all $i,j \in I$. Then $(\E,\S)$ is a small realized limit sketch satisfying $\Mod(\S) \simeq \Mod(X,\O_X)$. Now \cref{lfpue} implies that $\Mod(X,\O_X)$ is cocomplete and that cocontinuous functors $\Mod(X,\O_X) \to \C$ into a cocomplete category $\C$ correspond to $\S^{\op}$-models in $\C$, which are given by a $\C$-valued cosheaf $N$ on $X$ together with an $\O_X(U)$-comodule structure on each $N(U)$ such that for all $V \subseteq U$ the restriction $N(V) \to N(U)$ is $\O_X(V)$-colinear.
 
If $X$ is a scheme, a more natural category to consider is the category $\Qcoh(X)$ of quasi-coherent $\O_X$-modules. Limit and colimit sketches for this category will be studied in future work.
\end{ex}
 
\begin{ex}\label{preord}
Let $\Pre$ denote the category of preorders (a similar discussion works for the category $\Pos$ of partial orders). We would like to describe the category $\Hom_{\c}(\Pre,\C)$, where $\C$ is cocomplete. Since $\Pre$ is locally finitely presentable, there is a small finite limit sketch $\S$ with $\Mod(\S) \simeq \Pre$, which we can make more specific as follows: We start with two objects $R,X$ and two morphisms
\[r_1,r_2 : R \rightrightarrows X.\]
We add the monomorphism cone of $(r_1,r_2)$ (\cref{monocone}) to the sketch. This ensures that for a model $M$ the map $(M(r_1),M(r_2)) : M(R) \to M(X)^2$ becomes a binary relation. Next, we add a morphism $i : X \to R$ with $r_1 \circ i = \id_X = r_2 \circ i$, which ensures reflexivity of the relation. The compositions $i \circ r_1$, $i \circ r_2$ are some idempotent endomorphisms of $R$. For transitivity, we add another distinguished cone
\[\begin{tikzcd}
 & T \ar{dl}[swap]{\pr_1} \ar{dr}{\pr_2} & \\
R \ar{dr}[swap]{r_2} && R \ar{dl}{r_1} \\ 
 & X & 
\end{tikzcd}\]
as well as a morphism $c : T \to R$ satisfying $r_1 \circ c = r_1 \circ \pr_1$ and $r_2 \circ c = r_2 \circ \pr_2$. This defines all compositions. By construction of this sketch $\S$ we have $\Mod(\S) \simeq \Pre$. Now, \cref{lfpue} implies that
\[\Hom_{\c}(\Pre,\C) \simeq \CoPre(\C)\]
is the category of \emph{copreorders} internal to $\C$, which consist an epimorphism $X \sqcup X \twoheadrightarrow R$ in $\C$, a corelation, that is coreflexive and cotransitive. By \cref{modyon} this means that for every $T \in \C$ the induced relation $\Hom(R,T) \hookrightarrow \Hom(X,T)^2$ is reflexive and transitive. The universal copreorder in $\Pre$ is the evident epimorphism $\{0\} \sqcup \{0\} \twoheadrightarrow \{0 < 1\}$. For specific examples of $\C$ we can make the equivalence even more concrete. For example, we have
\begin{align*}
\Hom_{\c}(\Pre,\Set) & \simeq \Mono(\Set),\\
\Hom_{\c}(\Pos,\Set) & \simeq \Set.
\end{align*}
For the first equivalence we need to classify the copreorders in $\Set$. This has been done by Lumsdaine \cite{lumsdaine} in a much more general setting, but let us present the quick proof in our situation. Let $r : X \sqcup X \twoheadrightarrow R$ be a coreflexive (and cotransitive) corelation on a set $X$. Let $\sim$ be the equivalence relation on $X \sqcup X \cong X \times \{0,1\}$ corresponding to $r$. Since $r$ is coreflexive, $(x,t) \sim (x',t')$ implies $x=x'$. It follows that $\sim$ is completely determined by the subset $\{x \in X : (x,0) \sim (x,1)\} \subseteq X$. Conversely, if $A \subseteq X$ is any subset, we define an equivalence relation $\sim$ on $X \times \{0,1\}$ by $(x,t) \sim (x',t') \iff x=x' \wedge (t=t' \vee x \in A)$, and the induced relation $\leq$ on the sets $\Hom(X,Y)$ is the following: For $f,g : X \to Y$ we have $f \leq g$ if and only if $f|_A = g|_A$. This is actually an equivalence relation, so that we even get an equivalence corelation on $X$. Thus, copreorders on $X$ are parametrized by subsets $A \subseteq X$, and they are automatically cosymmetric. A morphism of copreorders $(X,A) \to (Y,B)$ is a map $\varphi : X \to Y$ with $\varphi(A) \subseteq B$. This shows the equivalence $\CoPre(\Set) \simeq \Mono(\Set)$. Since the mentioned relation on $\Hom(X,Y)$ is antisymmetric for all sets $Y$ if and only if $A=X$, this equivalence restricts to $\CoPos(\Set) \simeq \Set$. Let us also mention that for a subset $A \subseteq X$ the corresponding cocontinuous functor $\Pre \to \Set$ is given by the composition $\Pre \to \Eq \to \Set$, where $\Eq$ is the category of sets with an equivalence relation, $\Pre \to \Eq$ is the reflector mapping a preorder $(P,\leq)$ to $(P,\sim)$ where $\sim$ is the equivalence relation generated by $\leq$, and $\Eq \to \Set$ maps $(P,\sim)$ to $(P \times X)/((p,a) = (p',a) \text{ if } p \sim p',\, a \in A)$.
\end{ex}

\begin{ex}
We can adjust the construction in \cref{preord} to a limit sketch that models the category $\Cat$ of small categories. Namely, $R$ and $X$ are thought of as the generic sets of morphisms and objects: $r_1,r_2 : R \rightrightarrows X$ are source and target, $i : X \to R$ is the identity and $c : T \to R$ is the composition. We omit the monomorphism cone of $(r_1,r_2)$. But we add other morphisms and another fiber product to ensure unitality and associativity. We thus obtain the well-known notion of an internal (co)category, and \cref{lfpue} implies
\[\Hom_{\c}(\Cat,\C) \simeq \CoCat(\C).\]
Lumsdaine has shown that $\CoCat(\C) \simeq \Mono(\C)$ if $\C$ is a pretopos \cite{lumsdaine}.
\end{ex}

We will now look at infinitary examples.

\begin{ex}
Consider the category $\SupLat$ of suplattices (posets in which every subset has a supremum) and sup-preserving maps. It is the category of algebras of the covariant power set monad $(P,\eta,\mu)$ on $\Set$ with unit $\eta(x)=\{x\}$ and multiplication $\mu(A)=\bigcup A$. When $\C$ is a category with products, a $\C$-valued model of the corresponding infinitary Lawvere theory is an object $X \in \C$ together with morphisms $\sup_I : X^I \to X$ for all sets $I$ such that the following conditions hold (which we write down in the language of generalized elements):
\begin{enumerate}
\item For non-empty sets $I$ and a generalized element $x \in X$ we have $\sup_I (x)_{i \in I} = x$.
\item For a map $f : I \to J$ and a generalized element $x \in X^I$ we have the generalized transitivity law $\sup_I(x)=\sup_J\bigl((\sup_{f^{-1}(j)}(x|_{f^{-1}(j)}))_{j \in J}\bigr)$.
\end{enumerate}
Notice that (1) and (2) imply that $\sup$ is commutative in the obvious sense. Now \cref{monue} states that $\SupLat$ is the universal cocomplete category $\C$ with an internal \emph{cosuplattice}, which is an object $X \in \C$ equipped with morphisms $X \to I \otimes X$ satisfying the corresponding dual conditions. Coproducts coincide with products in $\SupLat$ -- even infinite ones. As in the case of abelian groups (\cref{watts}) one can easily deduce that every suplattice has a unique cosuplattice structure, which implies
\[\Hom_{\c}(\SupLat,\SupLat) \simeq \SupLat.\]
Every suplattice $A$ hence induces cocontinuous functor $- \otimes A : \SupLat \to \SupLat$, and it is well-known that $(\SupLat,\otimes)$ is a symmetric monoidal category.
\end{ex}

\begin{ex} \label{caba}
The contravariant power set functor $P : \Set^{\op} \to \Set$ is monadic \cite[Theorem 5.1.1]{barrwells} with left adjoint $P^{\op} : \Set \to \Set^{\op}$. The monad is a bit easier to understand using the category $\CABA$ of complete atomic boolean algebras. In fact, if we equip a power set $P(X)$ with the usual structure of a boolean algebra $\tilde{P}(X)$, we obtain an equivalence of categories $\smash{\tilde{P} : \Set^{\op} \to \CABA}$ \cite[Introduction]{johnstone}. The monadic functor therefore reduces to the forgetful functor $U : \CABA \to \Set$, whose left adjoint is $\smash{\tilde{P} \circ P^{\op}}$. Manes explains in \cite[Example 3.46]{manes2} directly why $\smash{\tilde{P}(P(X))}$ is the free CABA on a set $X$. Now \cref{internalalgebra} allows us to define CABAs internal to any complete category $\C$. This might be surprising, since a priori \emph{atomic} is not an algebraic condition, but actually it is equivalent to \emph{completely distributive} by \cite[Proposition VII.1.16]{johnstone}, and this is an algebraic condition. These are objects $B \in \C$ equipped with natural maps $P(P(X)) \to \Hom_{\C}(B^X,B)$ satisfying two axioms. See \cite[Formula 1.5.22]{manes1} for how the induced complete boolean algebra structure looks like. We now have the dual version of a \emph{coatomic cocomplete coboolean coalgebra} as well, and \cref{monue} implies that for every cocomplete category $\C$ there are equivalences
\[\Hom_{\c}(\Set^{\op},\C) \simeq \Hom_{\c}(\CABA,\C) \simeq \CoCABA(\C).\]
Equivalently, for every complete category $\C$ we have
\[\Hom^{\c}(\Set,\C) \simeq \CABA(\C),\]
showing that $\Set$ is the universal complete category with an internal CABA.
\end{ex}

\begin{ex}
The forgetful functor $\CompHaus \to \Set$ from compact Hausdorff spaces is monadic \cite[Proposition 3.17]{barrwells} with respect to the ultrafilter monad $\beta$. By \cref{monue} cocontinuous functors $\CompHaus \to \C$ are classified by objects $X \in \C$ equipped with monad maps $\beta \to \Hom(X,(-) \otimes X)$, which one can think of as an ultrafilter coconvergence on $X$. The universal example is $1 \in \CompHaus$ with the canonical isomorphisms $\beta(I) \to \Hom(1,I \otimes 1)$ (notice that $I \otimes 1$ is a coproduct in $\CompHaus$, not in $\Top$, and its underlying set is $\beta(I)$).
\end{ex}

\begin{ex}
The non-existence of a free complete lattice on three generators \cite[Theorem 1]{hales} implies that the concrete category of complete lattices is not monadic over $\Set$, so that our theory cannot describe cocontinuous functors on it.
\end{ex}


\section{Topological space objects} \label{sec:topspace}

Our next goal is to describe a specific limit sketch that models the category $\Top$. By \cref{internalize} this yields in particular the notion of a topological space internal to any complete category. But in this section we will go the other way round and give this internal definition first.

Let us mention that Burroni \cite{burroni1} has given a large mixed sketch that models $\Top$. A summary can be found in \cite{burroni2}. His construction is based on filter convergence. By axiomatizing neighborhood systems instead, Martins-Ferreira has given a characterization of topological spaces in terms of \emph{spacial fibrous preorders} \cite{ferreira}, and these can be internalized in any category with finite limits. For elementary topoi the definition of topological space objects in terms of open subobjects is a straight forward adaptation for the case of $\Set$, see for instance Macfarlane \cite{macfarlane}, Moerdijk \cite{moerdijk} and Stout \cite{stout}, and in fact later in \cref{sec:topop} we will give a similar definition.

Our approach uses the notion of \emph{net convergence} instead of filter convergence (which is more intuitive, since nets are just generalized sequences) and is closely related (but not identical) to Edgar's description of $\Top$ as a category of equational multialgebras \cite{edgar}. For this reason \cref{topob} and similarly \cref{topsketch} below must be known to (some) category theorists (see also \cite[Proposition 25]{guitart}), but the objective of this section is to make it more accessible and well-known.
 
\begin{rem} \label{netchar}
Recall \cite[Chapter 2]{kelley} that a net inside a set $X$ is a map $x : P \to X$ from the underlying set of a directed set $(P,\leq)$ (which we usually abbreviate by $P$) to $X$, and for a topological space $(X,\tau)$ it is said to converge to a point $s \in X$ when for every open set $s \in U \subseteq X$ there is some $p \in P$ such that $x$ maps $P_{\geq p} = \{q \in P : q \geq p\}$ into $U$. In that case we write $x \to s$. A subnet of $x : P \to X$ is a net of the form $x \circ h : Q \to X$ where $h : Q \to P$ is a cofinal map (not assumed to be monotonic, we just require that for all $p \in P$ there is some $q \in Q$ with $h(Q_{\geq q}) \subseteq P_{\geq p}$). A subset $A \subseteq X$ is closed if and only if for every convergent net $x \to s$ with values $x_p \in A$ we also have $s \in A$ \cite[Chapter 2, Theorem 6]{kelley}.
Hence, the topology is determined by the net convergence relation. Also, a map between topological spaces is continuous if and only if it preserves the net convergence relation \cite[Chapter 3, Theorem 1]{kelley}. The crucial observation is that it is possible to classify those net convergence relations induced by topologies: If $X$ is a set and $\to$ is a relation from nets in $X$ (defined on any directed sets) to points of $X$, then $\to$ is the net convergence relation induced by a (unique) topology on $X$ if and only if the following four axioms hold:
\begin{enumerate}
\item \emph{Constant nets.} For $s \in X$ we have $(s)_{p \in P} \to s$.
\item \emph{Subnets.} If $x \to s$, then $y \to s$ for every subnet $y$ of $x$.
\item \emph{Locality.} If $x$ is a net in $X$ and $s \in X$ is a point such that every subnet $y$ of $x$ has a subnet $z$ with $z \to s$, then $x \to s$.
\item \emph{Iterations.} Let $P$ be a directed set, and for each $p \in P$ let $Q_p$ be a directed set. Assume that $x : \coprod_{p \in P} Q_p \to X$ is a map such that for every $p \in P$ we have $(x_{p,q})_{q \in Q_p} \to s_p$ for some $s_p \in X$, and that we have $(s_p)_{p \in P} \to s$ for some $s \in X$. Then we also have $(x_{p,f(p)})_{(p,f)} \to s$ for the net indexed by the product $P \times \prod_{p \in P} Q_p$.
\end{enumerate}
This result is due to Kelley \cite[Chapter 2, Theorem 9]{kelley}, but a similar version (for $T_1$-spaces and with a stronger notion of subnet) was proven long before by Birkhoff \cite[Theorem 7]{birkhoff}. Actually, the third axiom can be strengthened, as can be seen in Kelley's proof: If $x : P \to X$ is a net and $s \in X$ is a point such that for every cofinal subset $Q \subseteq P$ the subnet $x|_Q : Q \to X$ has a subnet $z$ with $z \to s$, then $x \to s$. We will use this axiom instead because it requires less parameters and it is actually necessary to build the fiber product $F$ in the definition below.
\end{rem}
 
We have just described a category of \emph{sets} equipped with a good notion of net convergence that is isomorphic to $\Top$. By replacing sets with objects of a category, we can internalize the notion of a topological space as follows.
 
\begin{defi} \label{topob}
Let $\C$ be a complete category. A \emph{topology} or \emph{net convergence relation} on an object $X \in \C$ consists of a family of monomorphisms
\[\iota : C(P,X) \hookrightarrow X^P \times X\]
for directed sets $P$ subject to the following four axioms:
\begin{enumerate}
\item \emph{Constant nets.} For every directed set $P$ there exists a (necessarily unique) morphism $X \to C(P,X)$ such that
\[\begin{tikzcd}[column sep=20pt]
& X \ar{dr}{\Delta} \ar[dashed]{dl}[swap]{\exists}  & \\
C(P,X) \ar[hook,swap]{rr}{\iota} &&   X^P \times X
\end{tikzcd}\]
commutes. Here, $ \Delta$ is the diagonal morphism.
\item  \emph{Subnets.} For every cofinal map $h : Q \to P$ of directed sets there is a (necessarily unique) morphism $C(P,X) \to C(Q,X)$ such that the diagram
\[\begin{tikzcd}[column sep=35pt]
C(P,X) \ar[hook]{r}{\iota} \ar[dashed]{d}[swap]{\exists} & X^P \times X \ar{d}{X^h \times \id_X} \\
C(Q,X) \ar[hook,swap]{r}{\iota} &   X^Q \times X
\end{tikzcd}\]
commutes.
\item \emph{Locality.} Let $P$ be a directed set, and for each cofinal subset $Q \subseteq P$ let $h_Q : R_Q \to Q$ be some cofinal map. Consider the fiber product (with respect to the obvious morphisms)
\[F \coloneqq \left(\prod_{Q \subseteq P \text{ cofinal}} C(R_Q,X) \right) \times_{\prod\limits_{Q \subseteq P \text{ cofinal}} (X^{R_Q} \times X)} (X^P \times X).\]
There is a (necessarily unique) morphism $F \to C(P,X)$ such that the diagram
\[\begin{tikzcd}[column sep=20pt]
& F \ar[hook]{dr}{\pr_2} \ar[dashed]{dl}[swap]{\exists} &  \\
C(P,X) \ar[hook,swap]{rr}{\iota} &&   X^P \times X
\end{tikzcd}\]
commutes.
\item  \emph{Iterations.} Let $P$ be a directed set, and for each $p \in P$ let $Q_p$ be a directed set. There is a (necessarily unique) morphism on the fiber product (which respect to the obvious morphisms to $X^P$) 
\[\left(\prod_{p \in P} C(Q_p,X)\right) \times_{X^P} C(P,X) \to C\left(P \times \prod_{p \in P} Q_p,X\right)\]
such that the diagram
\[\begin{tikzcd}
\left(\prod\limits_{p \in P} C(Q_p,X)\right) \times_{X^P} C(P,X)  \ar[hook]{r}{\iota} \ar[dashed]{d}[swap]{\exists}
& \prod\limits_{p \in P} (X^{Q_p} \times X) \times_{X^P} (X^P \times X) \ar{d}{\vartheta}
\\
C\left(P \times \prod\limits_{p \in P} Q_p,X\right) \ar[hook,swap]{r}{\iota}
&  X^{P \times \prod_{p \in P} Q_p} \times X
\end{tikzcd}\]
commutes, where the morphism $\vartheta$ is defined via generalized elements by
\[\vartheta\left(\bigl((x_{p,q})_{q \in Q_p},s_p\bigr)_{p \in P},\bigl((s_p)_{p \in P},s\bigr)\right) \coloneqq \bigl((x_{p,f(p)})_{(p,f) \in  P \times \prod_{p \in P} Q_p},s\bigr).\]
Thus, $\vartheta$ is defined by $\pr_2 \circ \vartheta = \pr_2 \circ \pr_2$ and $\pr_{(p,f)} \circ \pr_1 \circ \vartheta = \pr_{f(p)} \circ \pr_1 \circ \pr_p \circ \pr_1$.
\end{enumerate}
We call $(X,C)$ a \emph{net-based topological space object} in $\C$. Usually we omit the word \emph{net-based}. A morphism $(X,C) \to (Y,D)$ of topological space objects is a morphism $f : X \to Y$ such that for all directed sets $P$ there is a (necessarily unique) morphism $C(P,X) \to D(P,Y)$ such that the diagram
\[\begin{tikzcd}[column sep=35pt]
C(P,X) \ar[hook]{r}{\iota} \ar[dashed]{d}[swap]{\exists} & X^P \times X \ar{d}{f^P \times f} \\
D(P,Y) \ar[hook,swap]{r}{\iota} &   Y^P \times Y
\end{tikzcd}\]
commutes. This defines a category $\Top(\C)$ with a forgetful functor $\Top(\C) \to \C$, $(X,C) \mapsto X$. 
\end{defi}

\begin{ex}
We record some special cases of this definition.
\begin{enumerate}
\item We have $\Top(\Set) \simeq \Top$ by \cref{netchar}.
\item If $\C$ is the category of $T$-algebras for some algebraic theory $T$, then $\Top(\C)$ is the category of topological $T$-algebras. For instance, $\Top(\Grp) \cong \Grp(\Top)$ is the category of topological groups, and $\Top(\Set_*) \cong \Top_*$ is the category of pointed topological spaces.
\item More generally, for every limit sketch $\S$ we have $\smash{\Top(\Mod(\S)) \cong \Mod_{\Top}(\S)}$ (since limits commute with limits). Thus, \cref{sheaves} implies for every topological space $S$ that $\Top(\Sh(S)) \cong \Sh(S,\Top)$ is the category of $\Top$-valued sheaves on $S$.
\item The iterated category $\Top(\Top)$ is closely related to the category $\BiTop$ of bitopological spaces (sets equipped with two topologies, without any compatibility relation, and maps that are continuous with respect to both topologies). Namely, $\BiTop$ is the full subcategory of those $(X,C) \in \Top(\Top)$ for which $C(P,X) \hookrightarrow X^P \times X$ is a \emph{regular} monomorphism in $\Top$, i.e.\ an embedding. Since we can always replace $C(P,X)$ by the image of $C(P,X) \hookrightarrow X^P \times X$, $\BiTop$ is a coreflective subcategory of $\Top(\Top)$.
\end{enumerate}
\end{ex}

\begin{rem} \label{observations}
Let us gather some easy observations.
\begin{enumerate}
\item The monomorphism $C(P,X) \hookrightarrow X^P \times X$ in \cref{topob} corresponds to a jointly monomorphic pair of morphisms $\{C(P,X) \to X^P,\, C(P,X) \to X\}$, which can be informally written as $(x \to s) \mapsto x$ and $(x \to s) \mapsto s$ for every net convergence relation $x \to s$ in $X$. Of course there is no harm to replace the monomorphism by a subobject $C(P,X) \subseteq X^P \times X$ (i.e.\ an isomorphism class of monomorphisms), but it is not necessary either. We could have written $X^{P \sqcup 1}$ instead of $X^P \times X$ as well.
\item If $\C$ is a complete category, then $\Top(\C)$ is complete as well, and the forgetful functor $\Top(\C) \to \C$ preserves limits. (This is because \cref{topob} uses the language of complete categories only, and limits commute with limits. We could also deduce this fact from \cref{limitsofmodels} after translating \cref{topob} into a limit sketch, which we will only do in the next section.) For example, the convergence of a net in a product of topological spaces can simply be tested componentwise. This description of a product space is much easier than the open sets in the product topology in the case $\C=\Set$.
\item Any continuous functor $\C \to \D$ between complete categories induces a continuous functor $\Top(\C) \to \Top(\D)$ such that the evident square commutes.
\item For a family of complete categories $(\C_i)_{i \in I}$ we have $\Top(\prod_{i \in I} \C_i) \cong \prod_{i \in I} \Top(\C_i)$.
\item Let $X \in \C$ be an object equipped with a class of morphisms $C(P,X) \to X^P \times X$ for all directed sets $P$. Then $(X,C) \in \Top(\C)$ if and only if for every test object $T \in \C$ the morphism $\Hom(T,C(P,X)) \to \Hom(T,X^P \times X) = \Hom(T,X)^P \times \Hom(T,X)$ defines a topology on the \emph{set} $\Hom(T,X)$ (in terms of net convergence). This follows from the Yoneda Lemma and \cref{netchar} (and again, we could deduce this from \cref{modyon} as soon as we have a limit sketch).
\item Generally speaking, one can follow the philosophy of categorical algebra (which apparently is not restricted to algebra) and generalize various constructions and notions from $\Top$ to $\Top(\C)$, where $\C$ is a well-behaved category. For example, we can define $(X,C) \in \Top(\C)$ to be Hausdorff when $C(P,X) \to X^P$ is a monomorphism for every directed set $P$, motivated by the characterization of usual Hausdorff topologies as those in which every net converges to at most one point \cite[Chapter 2, Theorem 3]{kelley}.
\item Any object $X \in \C$ can be equipped with the trivial topology $C(P,X) \coloneqq X^P \times X$ (the idea being that all nets converge to any point), also known as indiscrete topology. This provides a functor $\C \to \Top(\C)$ that is right adjoint to the forgetful functor. Discrete topologies are much more involved (which is the price to pay when not using open sets), see \cref{discrete} below.
\item If $X=1$ is the terminal object, then the trivial topology is the only topology on $X$.
\end{enumerate}
\end{rem}

Recall that a category $\C$ is called \emph{mono-complete} when for every object $X \in \C$ and every class of monomorphisms into $X$ their wide pullback over $X$ (i.e.\ their intersection) exists. For example, this holds when $\C$ is complete and wellpowered. Also recall the notion of a \emph{topological functor} \cite[Chapter 21]{ahs}, which is a faithful functor satisfying a certain lifting condition for (possibly large) structured sinks. The basic example is the forgetful functor $\Top \to \Set$. We generalize this example as follows.

\begin{prop}\label{topofunk}
Let $\C$ be a complete and mono-complete category. Then the forgetful functor $U : \Top(\C) \to \C$ is topological.
\end{prop}

\begin{proof}
Let $\{X \to U(S_i,C_i) = S_i\}_{i \in I}$ be a $U$-structured sink (where $I$ is possibly large). We want a net in $X$ to be convergent when its image in each $S_i$ is convergent. Thus, for a directed set $P$ let $C(P,X) \hookrightarrow X^P \times X$ be the intersection of the pullbacks
\[(X^P \times X) \times_{{S_i}^P \times S_i} C_i(P,S_i) \hookrightarrow X^P \times X.\]
We need to show that $C$ is a topology on $X$, and that a morphism $U(Y,D) \to X$ lifts to a morphism $(Y,D) \to (X,C)$ if and only if each composition $U(Y,D) \to X \to S_i$ does. This is actually a completely formal application of the universal property of $C(P,X)$ in $\C$. For example, the second topology axiom for $(S_i,C_i)$ asserts that every cofinal map $Q \to P$ induces a (unique) morphism $C_i(P,S_i) \to C_i(Q,S_i)$ over ${S_i}^P \times S_i \to {S_i}^Q \times S_i$, therefore a morphism
\[\smash{(X^P \times X) \times_{S_i^P \times S_i} C_i(P,Y_i) \to (X^Q \times X) \times_{{S_i}^Q \times S_i} C_i(Q,S_i)},\]
and hence a morphism $C(P,X) \to C(Q,X)$ over $X^P \times X \to X^Q \times X$. The other axioms can be proved in the same way. A morphism $Y = U(Y,D) \to X$ lifts to a morphism $(Y,D) \to (X,C)$ if and only if for every directed set $P$ the morphism $D(P,Y) \hookrightarrow Y^P \times Y \to X^P \times X$ factors through $C(P,X)$, which by definition of $C(P,X)$ means that for every $i \in I$ the composition $D(P,Y) \to {S_i}^P \times S_i$ factors through $C_i(P,S_i)$, which precisely means that $Y \to S_i$ lifts to a morphism $(Y,D) \to (S_i,C_i)$.
\end{proof}

The previous proof shows that, if $\C$ is just assumed to be complete, at least small structured sinks can be lifted.

\begin{cor} \label{topprop}
Let $\C$ be a mono-complete and complete category.
\begin{enumerate}
\item The forgetful functor $\Top(\C) \to \C$ has a left adjoint, the discrete topology.
\item If $\C$ is cocomplete, then $\Top(\C)$ is cocomplete too, and the forgetful functor $\Top(\C) \to \C$ preserves colimits.
\end{enumerate}
\end{cor}

\begin{proof}
This follows from \cref{topofunk} and general facts about topological functors, for which we refer to \cite[Chapter 21]{ahs}.
\end{proof}

The proof of \cref{topprop} does not give a concrete description of the discrete topology on an object of $\C$. At least this is possible if $\C$ is locally finitely presentable:

\begin{prop} \label{discrete}
Assume that $\C$ is a locally finitely presentable category. Then the forgetful functor $\Top(\C) \to \C$ has a left adjoint. It equips an object $X$ with the discrete topology defined by
\[C(P,X) \coloneqq \colim_{p \in P} X^P \times_{X^{P_{\geq p}}} X.\]
\end{prop}
 
\begin{proof}
Let $X \in \C$. We need to categorify the construction of the discrete topology on $X$ in terms of net convergence. We will divide the proof into several steps.

\textbf{Definition.} For directed sets $P$ and $p \in P$ let $C_p(P,X)$ be the fiber product of the diagonal morphism $X \to X^{P_{\geq p}}$ and the restriction morphism $X^P \to X^{P_{\geq p}}$. One can see $C_p(P,X)$ as the object of $P$-indexed nets in $X$ that become constant on $P_{\geq p}$. In particular, we have a monomorphism $C_p(P,X) \hookrightarrow X^P \times X$. For $p \leq q$ there is a morphism $C_p(P,X) \to C_q(P,X)$ over $X^P \times X$. We define
\[C(P,X) \coloneqq \colim_{p \in P} C_p(P,X),\]
which hence has a monomorphism to $X^P \times X$, since finite limits commute with filtered colimits in $\C$. Thus, $C(P,X)$ may be thought of as the object of nets in $X$ that are eventually constant. In the remainder of the proof, we will simplify the notation and just write
\[C_p(P,X) \subseteq C(P,X) \subseteq X^P \times X,\]
leaving the monomorphisms implicit, and use generalized elements throughout. Thus, the elements of $C_p(P,X)$ are those $(x,s) \in X^P \times X$ such that $x(q)=s$ for all $q \geq p$. Next, we verify the four topology axioms in \cref{topob}.

\textbf{Axiom 1.} For a generalized element $x \in X$ we trivially have $(\Delta(x),x) \in C_p(P,X)$ for any $p$ and hence $(\Delta(x),x) \in C(P,X)$.

\textbf{Axiom 2.} Let $h : Q \to P$ be a cofinal map of directed sets. Then $h^* \times \id : X^P \times X \to X^Q \times X$ maps $C(P,X)$ into $C(Q,X)$, since for every $p \in P$ we may choose $q \in Q$ with $h(Q_{\geq q}) \subseteq P_{\geq p}$, and then $h^* \times \id$ maps $C_p(P,X)$ into $C_q(Q,X)$.

\textbf{Axiom 3.} Let $(h_Q : R_Q \to Q)$ be a family of cofinal maps indexed by all cofinal subsets $Q \subseteq P$, and let $F \subseteq X^P \times X$ be the preimage (pullback) of $\prod_Q C(R_Q,X)$ under the canonical map $X^P \times X \to \prod_Q X^{R_Q} \times X$. Thus, $(x,s) \in X^P \times X$ belongs to $F$ if and only if for every cofinal subset $Q \subseteq P$ we have $(h_Q^*(x|_Q),s) \in C(R_Q,X)$. We need to show $F \subseteq C(P,X)$. Since $\C$ is locally finitely presentable, it suffices to prove $\Hom(T,F) \subseteq \Hom(T,C(P,X))$ for all finitely presentable objects $T$, i.e.\ that for all generalized elements of shape (i.e.\ domain) $T$ we have $(x,s) \in F \implies (x,s) \in C(P,X)$. So let $(x,s) \in F$, and assume $(x,s) \notin C(P,X)$. Then for all $p \in P$ we have $(x,s) \notin C_p(P,X)$, i.e.\ there is some $q \geq p$ with $x(q) \neq s$. This shows that $Q \coloneqq \{q \in P : x(q) \neq s\}$ is a cofinal subset of $P$. By assumption $(h_Q^*(x|_Q),s) \in C(R_Q,X)$ and hence, since the shape $T$ is finitely presentable, $(h_Q^*(x|_Q),s) \in C_r(R_Q,X)$ for some $r \in R_Q$. In particular, we have $s = h_Q^*(x|_Q)(r) = x(h_Q(r))$, which contradicts the definition of $Q$.

\textbf{Axiom 4.} Let $(Q_p)_{p \in P}$ be a family of directed sets indexed by a directed set $P$. We need to show that the morphism $\vartheta$ in \cref{topob} maps $\prod_{p \in P} C(Q_p,X) \times_{X^P} C(P,X)$ into $C(P \times \prod_{p \in P} Q_p,X)$. Again, we may work with generalized elements whose shape is finitely presentable. A generalized element of the fiber product has the form
\[\bigl(((x_{p,q})_{q \in Q_p},s_p)_{p \in P},((s_p)_{p \in P},s)\bigr),\]
where for each $p \in P$ there is some $g(p) \in Q_p$ (here we use the axiom of choice) with $x_{p,q}=s_p$ for all $q \in Q_{\geq g(p)}$, and there is some $p_0 \in P$ such that $s_p=s$ for all $p \in P_{\geq p_0}$. Then $(p_0,g) \in P \times \prod_{p \in P} Q_p$, and for all $(p,f) \geq (p_0,g)$ in that directed product set we have $x_{p,f(p)}=s_p$ because $f(p) \geq g(p)$ and $s_p=s$ because $p \geq p_0$, hence $x_{p,f(p)}=s$. This finishes the proof of $(X,C) \in \Top(\C)$.

\textbf{Universal Property.} To show the universal property, let $(Y,D) \in \Top(\C)$ and $f : X \to Y$ be a morphism. We need to verify that $f$ extends uniquely to a morphism $(X,C) \to (Y,D)$. This just means that for all directed sets $P$ the morphism $f^P \times f : X^P \times X \to Y^P \times Y$ maps $C(P,X)$ into $D(P,Y)$. By definition of $C(P,X)$ this means that for every $p \in P$ and every $(y,s) \in Y^P \times Y$ with $y(q)=s$ for all $q \in P_{\geq p}$ we have $(y,s) \in D(Y,P)$, i.e.\ that our categorical definition ``believes'' that eventually constant nets are convergent. This follows from the first and the third axiom: For every cofinal subset $Q \subseteq P$ the subnet $y|_Q$ has a subnet converging to $s$, namely the constant net $y|_{Q \cap P_{\geq p}}$.
\end{proof}

\begin{rem} \label{coprodtop}
If $\C$ is cocomplete, \cref{topprop} does not give a concrete description of colimits in $\Top(\C)$ either. Coequalizers are already hard to describe in terms of net convergence in the simplest example $\Top(\Set)$. But convergent nets in coproducts in $\Top(\Set)$ are easy to describe: They are eventually in one summand and converge there. We can categorify this as follows: Assume that $\C$ is a category with the following (topos-like) properties:
\begin{enumerate}
\item $\C$ is complete, has coproducts and filtered colimits,
\item coproduct inclusions are monomorphisms and pairwise disjoint,
\item monomorphisms are stable under coproducts and filtered colimits,
\item pullbacks are compatible with coproducts and filtered colimits,
\item $\C$ satisfies the IPC-property \cite[Definition 3.1.10]{kashi}, which asserts that (infinite) products commute with filtered colimits in a certain way.
\end{enumerate}
Then we can describe coproducts in $\Top(\C)$ as follows: Let $\{(X_i,C)\}_{i \in I}$ be a family of objects in $\Top(\C)$. We define a topology on $X \coloneqq \coprod_{i \in I} X_i$ as follows. Let $P$ be a directed set. For every $i \in I$ and $p \in P$ let
\[C^i_p(P,X) \coloneqq (X^P \times X_i) \times_{(X^{P_{\geq p}} \times X_i)} C(P_{\geq p},X_i).\]
For $p \leq q$ there is a canonical morphism $C^i_p(P,X) \to C^i_q(P,X)$ (since $P_{\geq q} \subseteq P_{\geq p}$ is cofinal), and we define
\[C^i(P,X) \coloneqq {\colim}_{p \in P} \, C^i_p(P,X),\]
and finally
\[C(P,X) \coloneqq \coprod_{i \in I} C^i(P,X).\]
The mentioned assumptions on $\C$ can be used to show that $C$ is a topology on $X$, namely they allow us to categorify the proof for $\C=\Set$. The universal property is easy to verify. We will omit the details here.
\end{rem}

The following lemma will be used to reduce the general classification of topologies on an object $X$ to the case that $X$ is an initial object.

\begin{lemma} \label{topslice}
Let $\C$ be a complete category, and let $C$ be a topology on $X \in \C$. Then the morphisms $X \to C(P,X)$ from the first topology axiom define a topology on $\id_X : X \to X$ in the coslice category $X / \C$.
\end{lemma}

\begin{proof}
The first topology axiom tells us that we have a monomorphism $C(P,X) \to X^P \times X$ in $X/\C$. For a cofinal map $Q \to P$ the induced morphism $C(P,X) \to C(Q,X)$ is a morphism in $X/\C$, i.e.\ 
\[\begin{tikzcd}[column sep=10pt]
 & X \ar{dr} \ar{dl} & \\
C(P,X) \ar{rr} && C(Q,X)
\end{tikzcd}\]
commutes, since this can be checked after composing with $C(Q,X) \hookrightarrow X^Q \times X$ and then using that $X^P \times X \to X^Q \times X$ is a morphism in $X/\C$. The third axiom follows in a similar way, since the fiber product $F$ there is built from morphisms in $X / \C$ and since $\pr_2 : F \to X^P \times X$ is a morphism in $X / \C$. The fourth axiom also follows since the morphism $\vartheta$ there is clearly a morphism in $X / \C$.
\end{proof}
 
We can dualize \cref{topob} and obtain the following notion of a cotopology (which is not related to the eponymous cotopological spaces studied in point-set topology \cite{aarts}, and also not to $\Q$-cotopological spaces studied in \cite{zhang}).
 
\begin{defi}
The category of \emph{cotopological space objects} in a cocomplete category $\C$ is defined by
\[\CoTop(\C) \coloneqq \Top(\C^{\op})^{\op}.\]
Thus, a \emph{cotopology} on an object $X \in \C$ consists of epimorphisms
\[(P \otimes X) \sqcup X \twoheadrightarrow C(P,X)\]
for all directed sets $P$ subject to four axioms (dual to those in \cref{topob}), where $P \otimes X$ denotes the copower $\coprod_{p \in P} X$.
\end{defi}

\begin{rem} \label{coobservations}
Each item from \cref{observations} can be dualized: The axioms of a cotopology are equivalent to the condition that for every $T \in \C$ the maps
\[\Hom(C(P,X),T) \hookrightarrow \Hom(( P \otimes X) \sqcup X,T) = \Hom(X,T)^P \times \Hom(X,T)\]
define a usual topology on the set $\Hom(X,T)$. Thus, $\Hom(X,-) : \C \to \Set$ lifts to a functor (which we also denote by)
\[\Hom(X,-) : \C \to \Top.\]
We notice that, in contrast to algebraic theories, such a lift is generally \emph{not} enough to define a cotopology on $X$. It should be a \emph{continuous} functor $\C \to \Top$, which is not automatic since $\Top \to \Set$ does not reflect limits; see also \cref{cotopset} below. Every object $X \in \C$ can be equipped with the trivial cotopology, consisting of the identities of $(P \otimes X) \sqcup X$. The resulting topology on each set $\Hom(X,Y)$ is the trivial topology. Also, $\CoTop(\C)$ is cocomplete, and the trivial cotopology induces a functor $\C \to \CoTop(\C)$ that is left adjoint to the forgetful functor. Any cocontinuous functor $\C \to \D$ induces a cocontinuous functor $\CoTop(\C) \to \CoTop(\D)$.
\end{rem}

\begin{ex} \label{Pinfty}
Let $(P,\leq)$ be a directed set. We endow the set $P \cup \{\infty\}$ (where $\infty$ is any element not contained in $P$) with the following topology: Every subset of $P$ is open (so we do \emph{not} use the order topology), and the sets of the form $P_{\geq p} \cup \{\infty\}$ (where $p \in P$) are a local base at $\infty$. Notice that for a topological space $T$ a continuous map $P \to T$ is the same as a $P$-indexed net in $T$, and that a continuous extension $P \cup \{\infty\} \to T$ is the same as a limit of that net (namely, the image of $\infty$).
\[\begin{tikzcd}[column sep=50pt, row sep=10pt]
P \ar{dr}{\text{net}} \ar{dd}[swap]{\subseteq} & \\
& T \\
P \cup \{\infty\} \ar{ur}[swap]{\text{\!\!convergent net}} &
\end{tikzcd}\]
This categorical characterization of limits is interesting in its own right. There is a canonical epimorphism $P \sqcup 1 \twoheadrightarrow P \cup \{\infty\}$, where $P \sqcup 1$ has the discrete topology, i.e.\ it is actually the copower $(P \sqcup 1) \otimes 1$. These epimorphisms constitute a cotopology on $1 \in \Top$. In fact, the induced topology on the space $\Hom(1,T) \cong T$ is just the topology of $T$, and the cotopology axioms follow by \cref{coobservations}. For example, any cofinal map $P \to Q$ extends to a continuous map $P \cup \{\infty\} \to Q \cup \{\infty\}$ since continuity at $\infty$ is precisely the cofinality of $ P \to Q$.
\end{ex}

We end this section with two classification results.

\begin{prop} \label{cotopset}
Let $X$ be a set. Then the trivial cotopology from \cref{coobservations} is the only cotopology on $X$. Hence, the functor $\Set \to \CoTop(\Set)$ from \cref{coobservations} is an isomorphism of categories.
\end{prop}

\begin{proof}
Let us first assume that $X=1$ is a terminal set. A cotopology on $1$ gives a lift of the functor $\Hom(1,-) \cong \id : \Set \to \Set$ to a continuous functor $\Set \to \Top$. Thus, every set $X$ is equipped with a topology, and every map $X \to Y$ becomes continuous. An example is the trivial topology on each set $X$, which in fact comes from the trivial cotopology on $1$. If, however, there is some set $X$ whose those topology is not trivial, there is some non-empty proper open subset $U \subseteq X$. Then every set $Y$ is equipped with the discrete topology: If $V \subseteq Y$ is any subset, there is a map $f : Y \to X$ with $f^{-1}(U)=V$, because we can just map all the elements of $V$ to some fixed element of $U$ and all the elements of $Y \setminus V$ to some fixed element of $X \setminus U$. Since $f$ is continuous, $V$ is open. But since the discrete topology does not give a continuous functor $\Set \to \Top$, this is not induced by a cotopology on $1$. This proves the claim for the case $X=1$.

The general case can be reduced to this as follows. If $X$ is any set with a cotopology, then by (the dual of) \cref{topslice} this defines a cotopology on $\id_X : X \to X$ in the slice category $\Set / X$. There is an equivalence of categories $\Set / X \to \Set^X$ mapping $ f : Y \to X$ to $(f^{-1}(x))_{x \in X}$, in particular mapping $\id_X$ to the family of terminal sets $(\{x\})_{x \in X}$, and the equivalence extends to an equivalence
\[\CoTop(\Set/X) \simeq \CoTop(\Set^X) \cong \CoTop(\Set)^X.\]
This means that the cotopology on $X$ is just the coproduct of certain cotopologies on $\{x\}$, where $x \in X$. But since these are trivial, so is the cotopology on $X$, and we are done.
\end{proof}

\begin{defi} \label{topydef}
A \emph{topological topology} $(X,\tau)$ as defined by Isbell \cite{isbell3} consists of a topological space $X$ and a topology $\tau$ on its set $\O(X)$ of open subsets such that the union and intersection operators
\[{\textstyle \bigcup} : \O(X)^I \to \O(X)\]
\[{\textstyle \bigcap} : \O(X)^J \to \O(X)\]
are continuous with respect to $\tau$ for every set $I$ and every finite set $J$. (Thus, $\O(X)$ becomes a topological frame.) A morphism $(X,\tau) \to (Y,\sigma)$ of topological topologies is a continuous map $f : X \to Y$ such that the induced map $f^* : \O(Y) \to \O(X)$ is a continuous map $ f^* : (\O(Y),\sigma) \to (\O(X),\tau)$. We obtain a category $\Top'$ with a forgetful functor $\Top' \to \Top$, $(X,\tau) \mapsto X$. The category $\Top'$ has been studied by Pedicchio \cite{pedicchio}.
\end{defi}

\begin{thm} \label{topychar}
There is an equivalence $\CoTop(\Top) \simeq \Top'$ over $\Top$.
\end{thm}

\begin{proof}
Let $\IS$ denote the Sierpinski space with underlying set $\{0,1\}$ and open sets $\emptyset,\{1\},\{0,1\}$. Any cotopology $C$ on $X \in \Top$ induces a continuous functor $\Hom(X,-) : \Top \to \Top$, so that in particular the set $\Hom(X,\IS)$ has a topology. We can transport it via the bijection $\O(X) \to \Hom(X,\IS)$, $U \mapsto \chi_U$ to a topology $\tau$ on $\O(X)$. If $I$ is a set, then the continuity of $\Hom(X,-) : \Top \to \Top$ yields $\O(X)^I \cong \Hom(X,\IS)^I \cong \Hom(X,\IS^I)$ in $\Top$, and the union operator $\bigcup : \O(X)^I \to \O(X)$ is represented by the continuous map $\IS^I \to \IS$ mapping $x \mapsto 1$ if and only if $x_i=1$ for some $i \in I$. Hence, $\bigcup : \O(X)^I \to \O(X)$ is continuous. If $J$ is a finite set, the same method shows that the intersection operator $\bigcap : \O(X)^J \to \O(X)$ is continuous by using the continuous map $\IS^J \to \IS$ that maps $x \mapsto 1$ if and only if $x_j=1$ for all $j \in J$. Therefore, $(X,\tau)$ is a topological topology. Explicitly, for open subsets $U_p \subseteq X$ ($p \in P$) and $U \subseteq X$ we have $(U_p)_{p \in P} \to U$ in $(\O(X),\tau)$ iff the corresponding continuous map $(P \otimes X) \sqcup X \to \IS$ extends to a continuous map on $C(P,X)$. It is obvious that every morphism of cotopologies is also a morphism of the associated topological topologies.

Conversely, let $(X,\tau)$ be a topological topology. Let $|~| : \Top \to \Set$ be the forgetful functor. For a directed set $P$ consider the set $(P \sqcup \{\infty\}) \times |X|$ together with the inclusion maps $\iota_p : |X| \to (P \sqcup \{\infty\}) \times |X|$ for $p \in P \sqcup \{\infty\}$. We define the following topology $\sigma$ on it: A subset $U$ is open iff we have $\iota_p^{-1}(U) \in \O(X)$ for every $p \in P \sqcup \{\infty\}$ and
\[(\iota_p^{-1}(U))_{p \in P} \to \iota_{\infty}^{-1}(U)\]
in the topological space $(\O(X),\tau)$. These open subsets are, in fact, closed under arbitrary unions and finite intersections, since each operator $\iota_p^{-1}$ is compatible with them and since these operations are continuous on $(\O(X),\tau)$ by assumption. This defines a topological space $C(P,X) \coloneqq ((P \sqcup \{\infty\}) \times |X|,\sigma)$. By construction each $\iota_p : X \to C(P,X)$ is continuous, so that they induce a continuous map $(P \sqcup 1) \otimes X \to C(P,X)$, which is in fact bijective (but not necessarily an isomorphism!). We claim that it is a cotopology: Since this is certainly true after applying the forgetful functor to $\Set$ (there we just have the trivial cotopology), we just need to check that the induced maps of sets in (the dual of) \cref{topob} are actually continuous. For the first axiom, we need to check that the codiagonal $C(P,X) \to X$, $(p,x) \mapsto x$ is continuous. If $V \subseteq |X|$ is open, its preimage $ U \coloneqq (P \sqcup \{\infty\}) \times V$ is open since $\iota_p^{-1}(U)=V$ for all $p$, and $(V)_{p \in P} \to V$. For the second axiom, we need to check that for every cofinal map $h : P \to Q$ the induced map $C(h,X) : C(P,X) \to C(Q,X)$, $(p,x) \mapsto (h(p),x)$ (with $h(\infty) \coloneqq \infty$) is continuous. If $U \subseteq |C(Q,X)|$ is open, then
\[\iota_p^{-1}\bigl(C(h,X)^{-1}(U)\bigr) = \iota_{h(p)}^{-1}(U)\]
is open for each $p \in P \sqcup \{\infty\}$, and we have
\[(\iota_{h(p)}^{-1}(U))_{p \in P} \to \iota_{\infty}^{-1}(U)\]
in $(\O(X),\tau)$ since this is a subnet of $\smash{(\iota_{q}^{-1}(U))_{q \in Q}}$, which shows that $C(h,X)^{-1}(U)$ is open. The other two axioms follow in the same way, each of them using the corresponding fact about convergent nets recalled in \cref{netchar}. (In fact, there is a much more conceptual argument that this works, but clearly a direct proof is also possible.) It is obvious that every morphism of topological topologies is also a morphism of the associated cotopologies.

The functors are pseudo-inverse to each other: Starting with a topological topology $(X,\tau)$, define $C(P,X)$ as above. For a family of open subsets $U_p \subseteq X$, $U_{\infty} \subseteq X$ we have $(U_p)_{p \in P} \to U_{\infty}$ with respect to $\tau$ iff $\{(p,x) : p \in P \sqcup \{\infty\},x \in U_p\}$ is open in $C(P,X)$ iff the map $C(P,X) \to \IS$ defined by $(p,x) \mapsto \chi_{U_p}(x)$ is continuous. This shows that $(X,\tau)$ is the topological topology associated to $C$. 
 
Conversely, let us start with a cotopology $C$ on $X$, define a topological topology $(X,\tau)$ as above. By \cref{cotopset} $C$ becomes the trivial cotopology after applying the forgetful functor $\Top \to \Set$, so that $(P \sqcup 1) \otimes X \to C(P,X)$ is a bijective continuous map. We only need to verify that the topology on $C(P,X)$ is the correct one. Since each $\iota_p : X \to C(P,X)$ is continuous, it follows that for every open set $U \subseteq C(P,X)$ the preimages $\iota_p^{-1}(U) \subseteq X$ are open as well. We need to show that $U$ is open iff
\[(\iota_p^{-1}(U))_{p \in P} \to \iota_{\infty}^{-1}(U)\]
with respect to $\tau$. By definition of $\tau$, this condition is equivalent to the continuity of the map $C(P,X) \to \IS$ defined by $\iota_p(x) \mapsto \chi_{\iota_p^{-1}(U)}(x)$. This is just the map $\chi_U$, and we are done.
\end{proof}


\section{A limit sketch for topological spaces} \label{sec:topsketch}

It is straight forward to transform \cref{topob} into the definition of a limit graph-sketch whose $\C$-valued models are topological space objects in $\C$. But for the reasons explained in \cref{compfree} and in order to apply \cref{main} we need to find a category-sketch instead. It is possible to do this by hand and define the sketch in terms of abstract objects and morphisms, but the definition of all the possible compositions is quite tedious, and in fact many more morphisms than those indicated by \cref{topob} appear. It is more efficient to define the sketch using the category of topological spaces itself.
 
\begin{defi} \label{topsketch}
We define a full subcategory $\E \subseteq \Top$ and a colimit sketch $(\E,\S)$ as follows:
\begin{enumerate}
\item Every discrete space $I$ belongs to $\E$, and the coproduct cocone $(i : 1 \to I)_{i \in I}$ belongs to $\S$.
\item For every directed set $P$ the space $P \cup \{\infty\}$ from \cref{Pinfty} belongs to $\E$, and the epimorphism cocone (\cref{monocone}) of $P \sqcup 1 \twoheadrightarrow P \cup \{\infty\}$ belongs to $\S$.
\item For every directed set $P$ and every family of cofinal maps $h = (h_Q : R_Q \to Q)$ indexed by cofinal subsets $Q \subseteq P$ and every colimit cocone
\[\begin{tikzcd}
& R_Q \cup \{\infty\} \ar{dr} \\
R_Q \sqcup 1 \ar{ur} \ar{dr} && F(P,h), \\
& P \sqcup 1 \ar{ur} & 
\end{tikzcd}\]
where $Q$ varies over all cofinal subsets of $P$, we have $F(P,h) \in \E$, and the colimit cocone belongs to $\S$. (We may also write $F(P,h)$ as a pushout of coproducts as in (the dual of) \cref{topob}, but we would like to keep the amount of necessary colimits here minimal.)
\item For every directed set $P$ and every family $Q = (Q_p)_{p \in P}$ of directed sets and every colimit cocone
\[\begin{tikzcd}
& Q_p \cup \{\infty\} \ar{dr} & \\
P \ar{ur}{\infty} \ar{dr} && G(P,Q), \\
 & P \cup \{\infty\} \ar{ur}
\end{tikzcd}\]
where $p$ varies over $P$, we have $G(P,Q) \in \E$, and the colimit cocone belongs to $\S$. (Again, we could write $G(P,Q)$ as a pushout of coproducts.)
\end{enumerate}
These are all objects of $\E$ resp.\ all cocones in $\S$. Finally, we define the \emph{limit sketch of topological spaces} $\T$ by $\S^{\op}$. Since $\S$ is realized, $\T$ is realized as well. The following result justifies our name for $\T$.
\end{defi}

\begin{thm} \label{topsketchthm}
Let $\T$ be the limit sketch of topological spaces from \cref{topsketch}. Then for every complete category $\C$ we have a natural equivalence of categories
\[\Mod_{\C}(\T) \simeq \Top(\C).\]
In particular, we have $\Mod(\T) \simeq \Top$.
\end{thm}

\begin{proof}
Let $(\E,\S)$ be the colimit sketch from \cref{topsketch}. There are four types of spaces and four types of distinguished cocones. Hence, a model $M : \E^{\op} \to \C$ of $\S^{\op}$ is a functor that is defined on the four types of spaces (and their morphisms) and maps the four types of colimit cocones to limit cones in $\C$. In particular, we obtain an object $M(1) \in \C$ and for every directed set $P$ a monomorphism $M(P \cup \{\infty\}) \hookrightarrow M(P \sqcup 1) \cong M(1)^{P \sqcup 1} = M(1)^P \times M(1)$. These define a topology on $M(1)$, essentially because of the cotopological space structure on $1 \in \Top$ from \cref{Pinfty}: For the first axiom, the unique map $P \cup \{\infty\} \to 1$ induces a morphism $M(1) \to M(P \cup \{\infty\})$. For the second axiom, any cofinal map $Q \to P$ induces a continuous map $Q \cup \{\infty\} \to P \cup \{\infty\}$ and hence a morphism $M(P \cup \{\infty\}) \to M(Q \cup \{\infty\})$. For the third axiom, the continuous map $P \cup \{\infty\} \to F(P,h)$ that extends $P \sqcup 1 \to F(P,h)$ induces a morphism $M(F(P,h)) \to M(P \cup \{\infty\})$, where $M(F(P,h))$ is isomorphic to the fiber product of products as in the third axiom of \cref{topob}. For the last axiom, the continuous map $(P \times \prod_{p \in P} Q_p) \cup \{\infty\} \to G(P,Q)$ gives a morphism $M(G(P,Q)) \to M((P \times \prod_{p \in P} Q_p) \cup \{\infty\})$, where $M(G(P,Q))$ is the fiber product of products as in the last axiom of \cref{topob}. For any morphism of models $M \to M'$ the morphism $M(1) \to M'(1)$ is clearly compatible with the topologies. This defines a functor $\Mod_{\C}(\T) \to \Top(\C)$.

Conversely, let $(X,C) \in \Top(\C)$, which we abbreviate by $X$ in this proof. For every $Y \in \C$ the functor $\Hom(Y,-) : \C \to \Set$ is continuous and therefore induces a functor, also denoted by $\Hom(Y,-) : \Top(\C) \to \Top(\Set)=\Top$. We claim that for every space $T \in \E$ there is an object $X^T \in \C$ satisfying the universal property
\[\Hom_{\C}(Y,X^T) \cong \Hom_{\Top}(T,\Hom_{\C}(Y,X)).\]
This adjunction shows in particular that $T \mapsto X^T$ will be an object of $\Mod_{\C}(\T)$. For the construction, we may use the same method as in the proof of \cref{tpex} and reduce the whole problem to the cases $T=1$ and $T=P \cup \{\infty\}$ for directed sets $P$. One easily checks that $X^1 \coloneqq X$ does the job, and the description of continuous maps on $P \cup \{\infty\}$ in \cref{Pinfty} shows that $X^{P \cup \{\infty\}} \coloneqq C(P,X)$ does the job. Clearly, any morphism $(X,C) \to (Y,D)$ induces morphisms $X^T \to Y^T$ which are natural in $T \in \E$. This finishes the construction of the functor $\Top(\C) \to \Mod_{\C}(\T)$.

The composition $\Top(\C) \to \Mod_{\C}(\T) \to \Top(\C)$ is the identity: If $(X,C) \in \Top(\C)$, the underlying object of its image is $X^1 = X$ and is equipped with the topology (what else?) $C$. To calculate the composition $\Mod_{\C}(\T) \to \Top(\C) \to \Mod_{\C}(\T)$, let $M : \E^{\op} \to \C$ be a $\C$-valued model of $\S^{\op}$. For $T \in \E$ there is a continuous map
\[T \cong \Hom_{\E}(1,T) \to \Hom_{\C}(M(T),M(1)),\]
inducing a $\C$-morphism
\[M(T) \to M(1)^T,\]
which is clearly natural in $T$. In order to show that it is an isomorphism, it suffices to check the cases $T=1$ and $T=P \cup \{\infty\}$, since these two types of spaces generate every object of $\E$ under distinguished colimits in $\S$ and both $M(-)$ and $M(1)^{(-)}$ are models of $\S^{\op}$. But these cases are obvious. 
\end{proof}

\begin{prop}\label{topsmall}
Let $\T$ be the limit sketch of topological spaces from \cref{topsketch}. Then every model of $\T$ is small (as defined in \cref{sec:small}).
\end{prop}

\begin{proof}
The rough idea of the proof is that net continuity of a map between two topological spaces does not have to be tested for arbitrary large directed sets, for which we just have to recycle the argument from \cite[Chapter 3, Theorem 1]{kelley}. Consider a $\T$-model $M : \E^{\op} \to \Set$ and the full subcategory $\E'$ of $\E$ that has as objects $1$ and the $D \cup \{\infty\}$ where $(D,\leq)$ is a directed set whose underlying set $D$ is a subset of the power set $P(M(1))$. Then $\E'$ is small. We claim that the restriction map $\Hom(M,N) \to \Hom(M|_{\E'^{\op}},N|_{\E'^{\op}})$ is bijective for every $\T$-model $N$, which is enough by \cref{smallchar} to prove the claim. Since $1 \in \E'$ and $\Hom(M,N) \to \Hom(M(1),N(1))$ is injective, the restriction map is injective. Now consider a morphism $M|_{\E'^{\op}} \to N|_{\E'^{\op}}$. Interpreting $M$ and $N$ as topological spaces with underlying sets $M(1)$, $N(1)$, this means that we are given a map $f : M(1) \to N(1)$ such that for all directed sets $D \subseteq P(M(1))$ and any $D$-indexed convergent net $(m_d)_{d \in D} \to m$ in $M(1)$ we have $(f(m_d))_{d \in D} \to f(m)$ in $N(1)$. We need to show that $f$ is continuous. If not, there is some $m \in M(1)$ and some open neighborhood $V$ of $f(m)$ such that for every open neighborhood $U$ of $m$ we have $f(U) \not\subseteq V$. The set $D \subseteq P(M(1))$ of open neighborhoods of $m$ is directed via reverse inclusion, and the axiom of choice allows us to choose a net $(m_U)_{U \in D}$ with $m_U \in U$ and $f(m_U) \notin V$. But then we have $(m_U)_{U \in D} \to m$ and $(f(m_U))_{U \in D} \not\to f(m)$. This contradiction shows that $f$ is indeed continuous (assuming the law of the excluded middle). And this means that $f$ is actually induced by a morphism $M \to N$.
\end{proof}

\begin{thm} \label{topue}
If $\C$ is any cocomplete category, then
\[\Hom_{\c}(\Top,\C) \simeq \CoTop(\C).\]
Hence, $\Top$ is the universal cocomplete category with a cotopological space object.
\end{thm}

\begin{proof}
By \cref{topsketchthm}, \cref{topsmall} as well as the general universal property in \cref{main} we have
\[\Hom_{\c}(\Top,\C) \simeq \Hom_{\c}(\Mod(\T),\C) \simeq \Mod_{\C}(\T^{\op}) \cong \Mod_{\C^{\op}}(\T)^{\op} \simeq \Top(\C^{\op})^{\op}. \qedhere\]
\end{proof}

\begin{rem}\label{explicittop}
Let us make the equivalence of \cref{topue} more explicit, following its proof. The universal cotopological space object in $\Top$ is the space $1$ with the cotopology consisting of the continuous maps $(P \sqcup 1) \otimes 1 \twoheadrightarrow P \cup \{\infty\}$ for directed sets $P$, see \cref{Pinfty}. Thus, any cocontinuous functor $F : \Top \to \C$ produces a cotopological space object in $\C$, namely $F(1)$ with the morphisms
\[(P \otimes F(1)) \sqcup F(1) \cong F((P \sqcup 1) \otimes 1) \twoheadrightarrow F(P \cup \{\infty\}).\]
Conversely, let $X$ be an object of $\C$ with a cotopology $(P \sqcup 1) \otimes X \twoheadrightarrow C(P,X)$, thus inducing topologies on the sets $\Hom(X,T)$ for $T \in \C$. For a topological space $M$ we then have a tensor product $X \otimes_{\T} M \in \C$ that is defined by the universal property that morphisms $X \otimes_{\T} M \to T$ in $\C$ correspond to continuous maps $M \to \Hom(X,T)$. Concretely, when $|M|$ denotes the underlying set of $M$, the tensor product $X \otimes_{\T} M$ is a (not necessarily regular) quotient of $X \otimes |M| = \coprod_{m \in |M|} X$ that identifies the net convergence in $M$ with the conet coconvergence in $X$. Specifically, there is a large colimit cocone
\[
\begin{tikzcd}[column sep=1pt, row sep=30pt]
& X \otimes (P \sqcup 1) \otimes |M|^{P \sqcup 1} \ar{rr}{\text{ev.}} && X \otimes |M| \ar[twoheadrightarrow]{dr} &  \\
X \otimes (P \sqcup 1) \otimes \Hom(P \cup \{\infty\},M) \ar[hookrightarrow]{ur}{\text{convergence}} \ar[twoheadrightarrow]{dr}[swap]{\text{coconvergence}} &&&& X \otimes_{\T} M, \\
& C(P,X) \otimes \Hom(P \cup \{\infty\},M) \ar[dashed]{urrr} && 
\end{tikzcd}\]
where $P$ runs through all directed sets, but this can be reduced to a small colimit exactly because of \cref{topsmall}. Then $X \otimes_{\T} - : \Top \to \C$ is a cocontinuous functor, which by definition is left adjoint to $\Hom(X,-) : \C \to \Top$. Every cocontinuous functor $\Top \to \C$ looks like that. In particular, we have seen again that $\Top$ is strongly compact.
\end{rem}

\begin{ex}
If $\C$ is a complete category, then \cref{topue} says that
\[\Hom_{\c}(\Top,\C^{\op}) \simeq \Top(\C)^{\op}.\]
For example, $\Hom_{\c}(\Top,\Set^{\op}) \simeq \Top^{\op}$ and $\Hom_{\c}(\Top,\Top^{\op}) \simeq \Top(\Top)^{\op}$. Specifically, the cocontinuous functor $\Top \to \Set^{\op}$ associated to a topological space $X$ is just the representable functor $\Hom(-,X) : \Top \to \Set^{\op}$, and a second topology on the underlying set of $X$ gives a lift to a cocontinuous functor $\Top \to \Top^{\op}$.
\end{ex}

The following classification has already been obtained by Isbell \cite[Theorem 1.3]{isbell3} with different topological arguments.
 
\begin{thm} \label{cocontop}
There is an equivalence of categories
\[\Hom_{\c}(\Top,\Set) \simeq \Set.\]
Every cocontinuous functor $\Top \to \Set$ is a coproduct of copies of the forgetful functor.
\end{thm}

\begin{proof}
Let $U : \Top \to \Set$ be the forgetful functor, and let $I : \Set \to \Top$ be the trivial topology functor, which is right adjoint to $U$. By \cref{cotopset} and \cref{topue} we have $\Hom_{\c}(\Top,\Set) \simeq \CoTop(\Set) \simeq \Set$. The pseudo-inverse maps a set $X$ first to the trivial cotopological set $X$ and then to the cocontinuous functor $\Top \to \Set$ that is left adjoint to $I \circ \Hom(X,-) : \Set \to \Set \to \Top$, which is $(X \otimes -) \circ U : \Top \to \Set \to \Set$.
\end{proof}

\begin{cor} \label{nocolimitsketch}
The category $\Top$ is not modeled by a colimit sketch.
\end{cor}

\begin{proof}
Take any continuous map of topological spaces that is bijective, but not an isomorphism. By \cref{cocontop} it gets mapped to an isomorphism by any cocontinuous functor $\Top \to \Set$. The claim follows from \cref{nec}.
\end{proof}

The next result has also been obtained already by Isbell \cite[Theorem 1.4]{isbell3} with different methods. Recall the category $\Top'$ from \cref{topydef}.

\begin{thm} \label{cocontop2}
There is an equivalence of categories
\[\Hom_{\c}(\Top,\Top) \simeq \Top'.\]
\end{thm}

\begin{proof}
This is a combination of \cref{topue} and \cref{topychar}.
\end{proof}

\begin{rem}
A topological space $X$ induces a cocontinuous functor $X \times - : \Top \to \Top$ if and only if $X$ is exponentiable if and only if $X$ is core-compact \cite{escardo}, and the resulting topological topology is $X$ with the Scott topology on $\O(X)$. So the category of core-compact spaces embeds fully faithfully into $\Hom_{\c}(\Top,\Top)$. If $X$ is core-compact, it carries the cotopology $X \times (P \sqcup 1) \twoheadrightarrow X \times (P \cup \{\infty\})$, and it follows easily that the topology on the spaces $\Hom(X,Y)$ can be described in terms of nets as follows: A net of continuous maps $(f_p : X \to Y)_{p \in P}$ converges to a continuous map $f : X \to Y$ if and only if for every $x \in X$ and every open neighborhood $f(x) \in V \subseteq Y$ there is some open neighborhood $U$ of $x$ and some $p_0 \in P$ such that $f_p(U) \subseteq V$ for all $p \geq p_0$. The functor $D \circ U : \Top \to \Top$ (where $D : \Set \to \Top$ is the discrete topology, and $U : \Top \to \Set$ is the forgetful functor) is cocontinuous and not of the form $X \times -$, and it corresponds to the trivial cotopological space $1$.
\end{rem}

\begin{rem}
There is a variant of \cref{topue} for Hausdorff spaces: We have seen in \cref{observations} how to define the category $\Haus(\C)$ of Hausdorff topological space objects internal to a complete category $\C$, such that $\Haus(\Set) \simeq \Haus$. This notion can be dualized: We let $\CoHaus(\C) \coloneqq \Haus(\C^{\op})^{\op}$ for a cocomplete category $\C$. Then \cref{topue} implies
\[\Hom_{\c}(\Haus,\C) \simeq \CoHaus(\C).\]
This is because cocontinuous functors $\Haus \to \C$ correspond to cocontinuous functors $\Top \to \C$ whose right adjoint factors through $\Haus$. Also, the limit sketch $\T$ from \cref{topsketch} can be modified to a limit sketch $\T_{\text{Haus}}$ with $\Mod_{\C}(\T_{\text{Haus}}) \simeq \Haus(\C)$: For a directed set $P$ the space $P \cup \{\infty\}$ is Hausdorff if and only if $P$ has no largest element. If $P$ has a largest element $p_{\infty}$, we therefore have to replace $P \cup \{\infty\}$ by its universal Hausdorff quotient, which is just the discrete space $P$ with $\infty \mapsto p_{\infty}$, since a $P$-indexed net in a Hausdorff space converges exactly to its value at $p_{\infty}$.
\end{rem}


\section{A limit sketch for the dual category of topological spaces} \label{sec:topop}

Our study of a limit sketch that models $\Top$ lead to a classification of cocontinuous functors on $\Top$. We also would like to classify \emph{continuous} functors on $\Top$. For this we need a limit sketch that models the dual category $\Top^{\op}$. Notice that $\Top^{\op}$ can be modeled by a realized limit sketch by our results on strongly compact categories (\cref{isbelllimit,saft}), but this sketch is too large to be useful. We will construct a better sketch that is based on the following description of $\Top^{\op}$.

\begin{rem} \label{topbool}
Let $\Frm$ denote the category of frames \cite[Definition II.1.1]{johnstone} and $\CABA$ denote the category of complete atomic boolean algebras. Let $U : \CABA \to \Frm$ be the forgetful functor. By \cite[Proposition 2]{adamekpedicchio} the category $\Top^{\op}$ is equivalent to the full subcategory of the comma category $\id_{\Frm} \downarrow U$ whose objects are the monomorphisms of frames
\[F \hookrightarrow U(B),\]
where $F \in \Frm$, $B \in \CABA$. In fact, $B$ is determined by its set of atoms $X$ with $B \cong P(X)$, and then $F \hookrightarrow P(X)$ is just a topology on $X$. A continuous map $(X,F) \to (Y,G)$ is the same as a morphism of boolean algebras $P(Y) \to P(X)$ that maps $G$ into $F$, which explains the dualization. This equivalence must be seen in the context of pointless topology, where we replace $\Top$ by $\Frm^{\op}$, the category of locales. The CABAs here are exactly what distinguishes topological spaces (spaces with points) from locales (spaces without points).
\end{rem}

This description motivates the following construction of a sketch for $\Top^{\op}$.

\begin{defi}\label{framesketch}
We already discussed that $\CABA$ is monadic over $\Set$ in \cref{caba}, and $\Frm$ is monadic over $\Set$ by \cite[Theorem II.1.2]{johnstone}. Hence, by \cref{sec:lawapp} there are infinitary Lawvere theories and $(\L_F,X)$ and $(\L_B,Y)$ with $\Frm \simeq \Mod(\L_F)$ and $\CABA \simeq \Mod(\L_B)$. There is a morphism of infinitary Lawvere theories $V : (\L_F,X) \to (\L_B,Y)$ that induces the forgetful functor $U : \CABA \to \Frm$ (any term in the theory of frames is also available in the theory of CABAs). We define a category $\E$ by starting with the coproduct category $\L_X \sqcup \L_Y$ and then adding the morphisms
\begin{align*}
\Hom_{\E}(X^I,Y^J) & \coloneqq \Hom_{\L_B}(Y^I,Y^J),\\
\Hom_{\E}(Y^J,X^I) & \coloneqq \Hom_{\L_F}(X^0,X^I).
\end{align*}
The compositions can be defined using $V$, for example
\begin{align*}
\Hom_{\E}(X^K,X^I) \times \Hom_{\E}(X^I,Y^J) & = \Hom_{\L_F}(X^K,X^I) \times \Hom_{\L_B}(Y^I,Y^J) \\
& \xrightarrow{V} \Hom_{\L_B}(Y^K,Y^I) \times \Hom_{\L_B}(Y^I,Y^J) \\
& \xrightarrow{\circ} \Hom_{\L_B}(Y^K,Y^J) \\
&= \Hom_{\E}(X^K,Y^J).
\end{align*}
We leave the other compositions to the reader. The identity of $Y \in \L_B$ corresponds to a morphism $\varphi : X \to Y$ in $\E$. The identities of $X^0 \in \L_F$ and $Y^0 \in \L_B$ correspond to morphisms $Y^0 \to X^0$ and $X^0 \to Y^0$ in $\E$  which are inverse to each other. (This explains why we defined the set $\Hom_{\E}(Y^J,X^I)$ to be non-empty, in fact we want both $X^0$ and $Y^0$ to be terminal in $\E$.) We define $\S$ as the class of cones in $\E$ consisting of the cones $(\pr_i : X^I \to X)_{i \in I}$, $(\pr_j : Y^J \to Y)_{j \in J}$ for all sets $I,J$ as well the monomorphism cone of $\varphi$ (\cref{monocone}). This defines a limit sketch $(\E,\S)$.
\end{defi}

\begin{prop} \label{topopsketch}
The limit sketch $(\E,\S)$ from \cref{framesketch} is realized and satisfies
\[\Mod(\S) \simeq \Top^{\op}.\]
\end{prop}

\begin{proof}
It is clear that $X^I$ is the $I$-fold power of $X$ in $\E$, likewise $Y^J$. Every morphism $X^I \to Y^J$ in $\E$ factors uniquely as $\varphi^I : X^I \to Y^I$ followed by a morphism $Y^I \to Y^J$ in $\L_B$. In particular, $\varphi^I$ is an epimorphism. The composition in $\E$ is defined in such a way that for every morphism $s : X^I \to X$ the diagram
\[\begin{tikzcd}[column sep=35pt]
X^I \ar{d}[swap]{s} \ar{r}{\varphi^I} & Y^I \ar{d}{V(s)} \\
X \ar{r}[swap]{\varphi}  & Y
\end{tikzcd}\]
commutes. Since $V : \L_F \to \L_B$ is fully faithful and $\varphi^I$ is an epimorphism, it follows that $\varphi$ is a monomorphism. Hence, $\S$ is realized and $\Mod(\S)$ is the category of triples $(F,B,\psi)$, where $F$ is a frame, $B$ is a complete atomic boolean algebra and $\psi : F \hookrightarrow U(B)$ is a monomorphism of frames. This category is equivalent to $\Top^{\op}$ by \cref{topbool}.
\end{proof}

\begin{lemma} \label{smallagain}
For the limit sketch $(\E,\S)$ from \cref{framesketch} every model is small.
\end{lemma}

\begin{proof}
This is an easy consequence of \cref{smallchar} and the following two observations: First, a map between the underlying sets of two frames $|F| \to |F'|$ is a morphism of frames $F \to F'$ if and only if infima of finite subsets of $F$ and suprema of subsets of $|F|$ are preserved -- so we do not need to consider families of elements indexed by arbitrary large sets, but we can bound them. Secondly, a map between the underlying sets of two complete atomic boolean algebras $|B| \to |B'|$ is a morphism $B \to B'$ if and only if complements of elements in $|B|$ and suprema of subsets of $|B|$ are preserved.
\end{proof}

\cref{topopsketch} yields an alternative definition of internal topological space objects.

\begin{defi}
Let $\C$ be complete category. The category $\Topf(\C)$ of \emph{frame-based topological space objects} internal to $\C$ is defined as the follows: Objects are triples $(F,B,\psi)$, where
\begin{itemize}
\item $F$ is an internal frame in $\C$,
\item $B$ is an internal complete atomic boolean algebra in $\C$,
\item and $\psi : F \hookrightarrow U(B)$ is a monomorphism of internal frames.
\end{itemize}
Here, $U(B)$ is the underlying internal frame of $B$. A morphism $(F,B,\psi) \to (F',B',\psi')$ is a pair of morphisms $F'\to F$, $ B' \to B$ such that the evident square commutes. Equivalently, it is just a morphism $B' \to B$ such that the composition $F' \hookrightarrow U(B') \to U(B)$ factors through $F \hookrightarrow U(B)$. Notice the dualization in the definition of morphisms.
\end{defi}

This definition is very similar to the topos-theoretic topological space objects appearing in the literature \cite{macfarlane,moerdijk,stout}, where $B$ is replaced by a power object of an object (which however is an internal boolean algebra only when the topos is boolean). In our case, we may imagine that $F$ is the frame of open subobjects, whereas $B$ is the boolean algebra of all subobjects of a \emph{non-existing} underlying object of the given space object. So with our definition there is actually no forgetful functor $\Topf(\C) \to \C$.

\begin{rem}
It is clear that the limit sketch $\S$ defined above satisfies
\[\Topf(\C) \simeq \Mod_{\C}(\S)^{\op},\]
and hence $\Topf(\Set) \simeq \Top$. In particular, $\Topf(\C)$ is \emph{cocomplete}. So the situation is quite contrary to the category $\Top(\C)$ from \cref{sec:topspace}, where limits are easy and colimits are hard. For $\Topf(\C)$ colimits are easy and limits are hard (maybe they do not exist at all, even when $\C$ is cocomplete). Every continuous functor $\C \to \D$ induces a \emph{cocontinuous} functor
\[\Topf(\C) \to \Topf(\D).\]
In particular, we get a cocontinuous functor $\Hom(A,-) : \Topf(\C) \to \Topf(\Set) \simeq \Top$ for every object $A \in \C$.
\end{rem}

\begin{thm} \label{contop}
Let $\C$ be a complete category. Then
\[\Hom^{\c}(\Top,\C) \simeq \Topf(\C)^{\op}.\]
Hence, up to this dualization, $\Top$ is the universal complete category with an internal frame-based topological space object.
\end{thm}

\begin{proof}
By \cref{topopsketch}, \cref{smallagain} and \cref{main} we get
\begin{align*}
\Hom^{\c}(\Top,\C) & \simeq \Hom^{\c}(\Mod(\S)^{\op},\C) \\
&  \cong \Hom_{\c}(\Mod(\S),\C^{\op})^{\op}  \\
&  \simeq \Mod_{\C^{\op}}(\S^{\op})^{\op} \\ & \cong \Mod_{\C}(\S) \\ & \simeq \Topf(\C)^{\op}. \qedhere
\end{align*}
\end{proof}

\begin{rem} \label{unfold}
Our proof can be unfolded and shows that the equivalence in \cref{contop} is induced by a universal frame-based topological space object internal to $\Top$. It can be described as follows. The underlying topological frame is the Sierpinski space $\IS$ together with the internal frame structure induced by the bijections $\Hom(T,\IS) \cong \O(T)$ and the frame structures on $\O(T)$ for every $T \in \Top$. The underlying topological CABA is the space $\mathbf{2}$ (trivial topology on two points) with the internal CABA structure induced by the bijections $\Hom(T,\mathbf{2}) \cong P(|T|)$ for every $T \in \Top$ and the usual CABA structure on the power sets. The inclusion $\IS \hookrightarrow \mathbf{2}$ is clear. A continuous functor $F : \Top \to \C$ thus corresponds to the object
\[F(\IS) \hookrightarrow F(\mathbf{2})\]
of $\smash{\Topf(\C)^{\op}}$. The inverse construction associates to every $\smash{(F \hookrightarrow U(B)) \in \Topf(\C)^{\op}}$ the continuous functor $\Top \to \C$ that is right adjoint to the functor $\smash{\C \to \Top \simeq \Topf(\Set)}$ mapping $A \in \C$ to $\Hom(A,F) \hookrightarrow \Hom(A,U(B))$.
\end{rem}

\begin{cor} \label{topsetc}
We have $\Hom^{\c}(\Top,\Set) \simeq \Top^{\op}$. Here, a topological space $X$ corresponds to the continuous functor $\Hom(X,-) : \Top \to \Set$.
\end{cor}

\begin{proof}
This follows from \cref{contop} and $\Topf(\Set) \simeq \Top$. Alternatively, we may just use that $\Top^{\op}$ is strongly compact and \cref{isbellchar}.
\end{proof}

Recall the category of topological topologies $\Top'$ from \cref{topydef}.
 
\begin{thm} \label{toptopcont}
The category $\Hom^{\c}(\Top,\Top)$ is equivalent to $\Top'^{\op}$.
\end{thm}
 
\begin{proof}
By \cref{contop} our task is to prove $\Topf(\Top) \simeq \Top'$. First, we prove that the forgetful functor $\Top \to \Set$ induces an equivalence $\CABA(\Top) \simeq  \CABA(\Set) \simeq \Set^{\op}$, i.e.\ that every topological CABA is $P(X)$ for some set $X$, where $P(X)$ is equipped with the trivial topology. (This is actually equivalent to \cref{cocontop}, but we give a different proof.) The underlying set CABA certainly has this form, and the only additional information we have is that the union operator $\bigcup : P(X)^I \to P(X)$, the finite intersection operator $\bigcap : P(X)^J \to P(X)$ and the complement operator $P(X) \to P(X)$ are continuous maps. If $X=1$ is a point, then $P(1) \cong \{0,1\}$ has only four topologies. The two Sierpinski topologies are invalid because the complement operator is not continuous for these. The discrete topology is not valid since then $\bigcup : \{0,1\}^{\IN} \to \{0,1\}$ is not continuous at the zero sequence. So the trivial topology is the only valid one, and in fact $\mathbf{2}$ is a topological CABA by \cref{unfold}. Now let $X$ be a general set. For $x \in X$ we equip $P(\{x\})$ with the subspace topology from the inclusion $P(\{x\}) \hookrightarrow P(X)$. It is easy to check that $P(\{x\})$ is again a topological CABA, so it carries the trivial topology by what we have seen. The CABA map $P(X) \to P(\{x\})$ induced by $\{x\} \hookrightarrow X$ is continuous since its composition with $P(\{x\}) \hookrightarrow P(X)$ is the continuous map $P(X) \to P(X)$, $A \mapsto A \cap \{x\}$. These induce a morphism of topological CABAs
\[\begin{tikzcd}
P(X) \ar{r} & \prod_{x \in X} P(\{x\}).
\end{tikzcd}\]
We claim that it is an isomorphism. In fact, the inverse map is the composition of continuous maps
\[\begin{tikzcd}
\prod_{x \in X} P(\{x\}) \ar[hook]{r} & \prod_{x \in X} P(X) \ar{r}{\bigcup} & P(X).
\end{tikzcd}\]
Since each $P(\{x\})$ carries the trivial topology, the same is true for $P(X)$. This finishes the proof of $\CABA(\Top) \simeq \Set^{\op}$. It follows that the objects of $\Topf(\Top)$ are monomorphisms of topological frames $\psi : F \hookrightarrow P(X)$, where $P(X)$ carries the trivial topology, which means that continuity of $\psi$ is automatic. So we end up with a topological topology on $X$.
\end{proof}

By \cref{adjointgame} we have $\Hom_{\c}(\Top,\Top) \simeq \Hom^{\c}(\Top,\Top)^{\op}$, so actually \cref{toptopcont} is equivalent to Isbell's \cref{cocontop2}. Thus, we obtain yet another, more simple proof of that theorem. 
 
\appendix


\section{Infinitary Lawvere theories} \label{sec:lawapp}

Lawvere theories are single-sorted finite product theories that provide a formalization of finitary algebraic theories \cite{lawvere,barrwells}. They can be seen as special cases of small limit sketches: The underlying category $\L$ has finite products and consists of the finite powers $1,X,X^2,\dotsc$ of a single object $X$. The distinguished cones of $\L$ are finite product cones $(\pr_i : X^n \to X)_{1 \leq i \leq n}$. We will denote the limit sketch also by $\L$. Thus, $\L$-models are functors $\L \to \Set$ that preserve finite products. They are usually called $\L$-algebras.

In this expository appendix we will look at an infinitary version of Lawvere theories, namely single-sorted product theories where products may have arbitrary sizes. Specifically, we will define them as special limit sketches and prove:
\begin{enumerate}
\item Infinitary Lawvere theories are equivalent to monads on $\Set$ (such that Lawvere theories correspond to finitary monads).
\item The category of models of an infinitary Lawvere theory is equivalent to the category of algebras of its monad.
\end{enumerate}
Both results are well-known and classical, usually attributed to Linton's paper on functorial semantics \cite{linton2}, and (1) was already announced in \cite{linton}, where Linton introduced infinitary Lawvere theories under the name ``varietal theories''. The goal of this appendix is to provide an elementary, beginner friendly (in particular, non-enriched), self-contained (in particular, using neither Beck's monadicty theorem nor Birkhoff's variety theorem), fully detailed and direct proof of (1) and (2). The author could not find such a proof in the literature.
 
In fact, Linton's theory developed in \cite{linton2} is very general, the equivalences (1) and (2) are never explicitly stated, and it is not straight forward how to derive them from the theory, although it is clear from the introduction of the proceedings that Linton had (a generalization of) these theorems in mind. Dubuc \cite{dubuc} was the first one to explicitly state and prove (1) and (2), even in the more general setting of enriched categories. He uses codensity monads and the proofs omit some details. Power \cite{power} also works in the enriched setting, omits routine verifications, only works in the finitary setting, and proves (1) with Beck's monadicity theorem. The enriched theory was subsequently generalized by Nishizawa and Power \cite{nishizawa}, allowing any arities in particular. Lucyshyn-Wright \cite{lucyshyn} takes the concept of arities even further and has given a very general and concise extension of the theory; (1) is a special case of his Theorem 11.8, and (2) is a special case of his Theorem 11.14. In their book on algebraic theories \cite{arv} Adámek, Rosický and Vitale give a beginner-friendly proof, which however is restricted to the finitary case and uses Birkhoff's variety theorem for the finitary case of (2). Manes' book on algebraic theories proves a small part of (2) in \cite[Lemma 1.5.36]{manes1} and leaves (a generalization of) the rest as a guided exercise \cite[Exercise 3.2.7]{manes1}.
 
\begin{defi} \label{lawdef}
An \emph{infinitary Lawvere theory} $(\L,X)$ consists of a category $\L$ (which we assume to be locally small, as always) and a distinguished object $X \in \L$ such that for every set $I$ a power $X^I$ exists and every object of $\L$ is isomorphic to such a power. We do not require that every object is \emph{equal} to some specified power since this would violate the principle of equivalence. The corresponding (possibly large) realized limit sketch consisting of the product cones $(\pr_i : X^I \to X)_{i \in I}$ for sets $I$ will also be denoted by $\L$. Thus, a $\C$-valued model of $\L$ is a functor $M : \L \to \C$ such that for every set $I$ the cone $(M(\pr_i) : M(X^I) \to M(X))_{i \in I}$ is a product cone in $\C$. We often abbreviate $(\L,X)$ by $\L$. 
\end{defi}

\begin{rem} \label{lawrem}
Let $(\L,X)$ be an infinitary Lawvere theory. The category $\L$ has all products, and we have a product-preserving functor $\Set^{\op} \to \L$, $I \mapsto X^I$, which is essentially surjective (and this is how Lawvere theories are often defined). Notice that $\C$-valued $\L$-models are just product-preserving functors $M : \L \to \C$. They can also be described as follows: We have an object $M(X) \in \C$, called the \emph{underlying object} of $M$, whose powers exist, and for every pair of sets $I,J$ a map
\[M : \Hom(X^I,X^J) \to \Hom(M(X)^I,M(X)^J)\]
which are compatible with compositions and projections. Equivalently, for all sets $I$ we have maps
\[M : \Hom(X^I,X) \to \Hom(M(X)^I,M(X))\]
which are compatible with compositions and projections. Explicitly:
\begin{enumerate}
\item For every set $I$ and every $i \in I$ we have $M(\pr_i : X^I \to X) = \pr_i : M(X)^I \to M(X)$.
\item If $f : X^I \to X$ is a morphism and $(g_i : X^J \to X)_{i \in I}$ is a family of morphisms in $\L$, then
\[M\bigl(f \circ (g_i)_{i \in I} : X^J \to X\bigr) = M(f) \circ (M(g_i))_{i \in I} : M(X)^J \to M(X).\]
\end{enumerate}
This is more close to the usual description of an algebraic object as an underlying object $M(X)$ together with $I$-ary operations $M(f) : M(X)^I \to M(X)$ for every operation symbol $f \in \Hom(X^I,X)$, but one can argue that the description as a functor preserving products is much simpler. Equations between the operations are just encoded in the category $\L$. A morphism of $\L$-models $M \to N$ is equivalent to a morphism $M(X) \to N(X)$ in $\C$ that is compatible with the operations in the obvious sense. The forgetful functor
\[U : \Mod_{\C}(\L) \to \C, ~ M \mapsto M(X)\]
is faithful and conservative. In particular, $\Mod_{\C}(\L)$ is locally small and hence a category in our sense (which also follows from \cref{locallysmall}). When $\C$ is complete, it is also complete by \cref{limitsofmodels}.
\end{rem}

\begin{rem} \label{finitary}
Finitary Lawvere theories can be extended to infinitary Lawvere theories as follows: Let $(\L,X)$ be a finitary Lawvere theory, so $\L$ has finite products and every object of $\L$ is isomorphic to some finite power of $X$. We define an infinitary Lawvere theory $(\L^{\infty},X)$ as follows: The objects are symbols $X^I$ for sets $I$, and the morphisms are defined by
\[\Hom_{\L^{\infty}}(X^I,X^J) \coloneqq \left(\colim_{E \subseteq I \text{ finite}} \Hom_{\L}(X^E,X)\right)^J.\]
Thus, for finite sets $I,J$ we have $\Hom_{\L^{\infty}}(X^I,X^J)=\Hom_{\L}(X^I,X^J)$. The composition is easy to write down. Models of $(\L^{\infty},X)$ are the same as models of $(\L,X)$. Because of this, the term \emph{general Lawvere theory} is perhaps more appropriate, also in light of their correspondence to \emph{general} monads in \cref{lawmon} below.
\end{rem}

\begin{defi}
There is a strict $2$-category $\Law_{\infty}$ (not locally small) of infinitary Lawvere theories defined as follows: A morphism $(\L,X) \to (\K,Y)$ is a functor $F : \L \to \K$ preserving products together with an isomorphism $h : F(X) \to Y$. (It seems to be more common to require $F(X)=Y$, but this violates the principle of equivalence.) We define a $2$-morphism $\alpha : (F,h) \to (G,k)$ to be a morphism of functors $\alpha : F \to G$ with $k \circ \alpha(X) = h$. It follows that $\alpha$ is unique (if it exists) and an isomorphism. So the hom-categories of $\Law_{\infty}$ are setoids, and $\Law_{\infty}$ can therefore be replaced by an equivalent $1$-category. Clearly, any morphism $(\L,X) \to (\K,Y)$ induces a continuous functor $\Mod(\K) \to \Mod(\L)$ over $\Set$.
\end{defi}

\begin{lemma} \label{freemodel}
Let $\L$ be an infinitary Lawvere theory. The forgetful functor $U : \Mod(\L) \to \Set$, $M \mapsto M(X)$ has a left adjoint $F : \Set \to \Mod(\L)$, namely $F(I) \coloneqq \Hom(X^I,-)$.
\end{lemma}

\begin{proof}
For any $M \in \Mod(\L)$ we have
\[\Hom\bigl(\Hom(X^I,-),M\bigr) \cong M(X^I) \cong M(X)^I = \Hom(I,M(X)),\]
where the first isomorphism is by the Yoneda Lemma.
\end{proof}

\begin{thm} \label{ismonadic}
For every infinitary Lawvere theory $\L$ the forgetful functor $\Mod(\L) \to \Set$ is monadic. Thus, there is an equivalence
\[\Mod(\L) \simeq \Alg(\T)\]
for a monad $\T$ on $\Set$. This equivalence maps representable $\L$-models onto free $\T$-algebras.
\end{thm}

\begin{proof}
We divide the proof into six parts.
 
\textbf{Description of the Monad.} First, we need to describe the monad $\T = (T,\eta,\mu)$ on $\Set$ induced by the adjunction $F \dashv U$ from \cref{freemodel}. The functor $ T : \Set \to \Set$ is defined by $T \coloneqq U \circ F$, that is
\[T(I) = \Hom_{\L}(X^I,X).\]
This is the set of $I$-ary operations of the Lawvere theory. If $ f : I \to J$ is a map, the induced map $T(f) : T(I) \to T(J)$ is defined by $T(f)(t) = t \circ X^f$, where $X^f : X^J \to X^I$ is defined by $\pr_i \circ X^f = \pr_{f(i)}$. The unit $\eta : \id_{\Set} \to T$ is defined by the projections
\[\eta_I : I \to \Hom_{\L}(X^I,X),~ i \mapsto \pr_i.\]
The multiplication $\mu : T^2 \to T$ is defined by $\mu \coloneqq U \varepsilon F$ for the counit $\varepsilon : FU \to \id_{\Mod(\L)}$. For a model $M$ the counit
\[\varepsilon_M : \Hom_{\L}\bigl(X^{M(X)},-\bigr) \to M\]
is induced via Yoneda by, what we call, the \emph{universal element}
\[u_M \in M\bigl(X^{M(X)}\bigr) \cong M(X)^{M(X)}\]
corresponding to the identity on $M(X)$. (It is worthwhile to mention that there is no such universal element in $M(X^{M(X)})$ in the finitary case, since $M(X)$ does not have to be a finite set.) Thus, $\varepsilon_M(f) = M(f)(u_M)$ for morphisms $f : X^{M(X)} \to X^J$. When $M=\Hom_{\L}(X^I,-)$ is a free model, let us just write $u_I$ for the universal element. It follows that $\mu$ is given by
\[\mu_I : \Hom_{\L}(X^{\Hom_{\L}(X^I,X)},X) \to \Hom_{\L}(X^I,X),~ f \mapsto f \circ u_I,\]
where the universal element
\[u_I : X^I \to X^{\Hom_{\L}(X^I,X)}\]
is characterized by
\[\pr_g \circ u_I = g\]
for all $g : X^I \to X$. This finishes the description of the monad.

\textbf{Properties of the universal element.} The universal element $u_M \in M\bigl(X^{M(X)}\bigr)$ has the following property: For every map $f : I \to M(X)$ we have
\[M(X^f)(u_M) = f\]
under the identification $M(X^I) \cong M(X)^I$. In particular, every map $t : I \to T(J)$ satisfies
\[X^t \circ u_J = (t_i)_{i \in I}.\]
Also notice that $u_I : X^I \to X^{T(I)}$ is natural in $I$: If $f : I \to J$ is a map, we have
\[u_I \circ X^f = X^{T(f)} \circ u_J,\]
which can be simply checked by composing with the projections $X^{T(I)} \to X$. We will use all of these properties in the following.

\textbf{Description of the comparison functor.} The canonical comparison functor
\[C : \Mod(\L) \to \Alg(\T)\]
maps, by definition, a model $M : \L \to \Set$ to the $\T$-algebra $(U(M), U(\varepsilon_M))$. Explicitly, this is the set $M(X)$ with the action
\[h_M : T(M(X)) = \Hom_{\L}\bigl(X^{M(X)},X\bigr) \xrightarrow{M} \Hom_{\Set}\bigl(M(X^{M(X)}),M(X)\bigr) \xrightarrow{\ev_{u_M}} M(X).\]
We need to prove that $C$ is an equivalence of categories (in fact, $C$ is an isomorphism of categories). Since $U$ is faithful, $C$ is faithful as well.

\textbf{Fullness of $C$.} To show that $C$ is full, let $M,N$ be two models and let $\beta : M(X) \to N(X)$ be a map that is compatible with the $\T$-actions, i.e.\ the diagram
\[\begin{tikzcd}
T(M(X)) \ar{r}{T(\beta)} \ar{d}[swap]{h_M} & T(N(X)) \ar{d}{h_N} \\
M(X) \ar{r}[swap]{\beta} & N(X)
\end{tikzcd}\]
commutes. We need to prove that for all $t : X^I \to X$, i.e.\ $t \in T(I)$, the diagram
\[\begin{tikzcd}
M(X)^I \ar{r}{\beta^I} \ar{d}[swap]{M(t)} & N(X)^I \ar{d}{N(t)} \\
M(X) \ar{r}[swap]{\beta} & N(X)
\end{tikzcd}\]
commutes. So let $m \in M(X)^I$, which is a map $m : I \to M(X)$. The first diagram applied to the element $T(m)(t) \in T(M(X))$ gives
\[h_N\bigl(T(\beta)(T(m)(t))\bigr) = \beta\bigl(h_M(T(m)(t))\bigr)\]
in $N(X)$. The left side is equal to
\begin{align*}
h_N\bigl(T(\beta \circ m)(t)\bigr) &= h_N\bigl(t \circ X^{\beta \circ  m}\bigr) \\
& = N\bigl(t \circ X^{\beta \circ  m}\bigr)(u_N) \\
& = N(t)\bigl(N(X^{\beta \circ m})(u_N)\bigr)\\
& = N(t)(\beta \circ m) \\
& = N(t)(\beta^I(m)).
\end{align*}
The right side is equal to
\begin{align*}
\beta\bigl(h_M(T(m)(t))\bigr) &= \beta\bigl(M(t \circ X^m)(u_M)\bigr) \\
&= \beta\bigl(M(t)(M(X^m)(u_M))\bigr) \\
&= \beta(M(t)(m)).
\end{align*}
This finishes the proof that $C$ is full.

\textbf{Surjectivity of $C$.} To show that $C$ is (essentially) surjective, let $(A,h : T(A) \to A)$ be a $\T$-algebra, which means $h : \Hom_{\L}(X^A,X) \to A$ is a map satisfying
\begin{align*}
h \circ \eta_A &= \id_A,\\
h \circ T(h) &= h \circ \mu_A.
\end{align*}
By our description of the monad, this means that for all $a \in A$ and $f \in T^2(A)$ we have
\begin{align*}
h(\pr_a) &= a,\\
h\bigl(f \circ X^h\bigr) &= h\bigl(f \circ u_A\bigr).
\end{align*}
We define a model $M \in \Mod(\L)$ with underlying set $M(X) \coloneqq A$ that maps a morphism $t \in T(I) = \Hom_{\L}(X^I,X)$ to the map $M(t) : A^I \to A$ that maps $a \in A^I = \Hom(I,A)$ to
\[M(t)(a) \coloneqq h(T(a)(t)) = h\bigl(t \circ X^a\bigr).\]
For a projection $\pr_i : X^I \to X$ we get
\[M(\pr_i)(a) = h\bigl(\pr_i \circ X^a\bigr) = h(\pr_{a(i)})=a(i),\]
so that $M(\pr_i)=\pr_i$. Next, for a morphism $ s : X^I \to X$ and a family of morphisms $(t_i : X^J \to X)_{i \in I}$, i.e.\ $s \in T(I)$ and $t \in \Hom(I,T(J))$, we need to prove
\[M(s) \circ (M(t_i))_{i \in I} = M(s \circ (t_i)_{i \in I})\]
as maps $A^J \to A$. So let $a \in A^J = \Hom(J,A)$. Define $f \in T^2(A)$ as the composition
\[f : X^{T(A)} \xrightarrow{X^{T(a)}} X^{T(J)} \xrightarrow{~X^t~} X^I \xrightarrow{~~s~~} X.\]
By assumption we have $h\bigl(f \circ X^h\bigr) = h\bigl(f \circ u_A\bigr)$. The left side is equal to
\begin{align*}
h\bigl(f \circ X^h\bigr) &= h\bigl(s \circ X^t \circ X^{T(a)} \circ X^h\bigr) \\
&= h\bigl(s \circ X^{h \circ T(a) \circ t}\bigr) \\
&= h\bigl(s \circ X^{i \,\mapsto\, h(T(a)(t_i))}\bigr)\\
&  = h\bigl(s \circ X^{i \,\mapsto\, M(t_i)(a)}\bigr)  \\
&= M(s)\bigl((M(t_i))_{i \in I}(a)\bigr).
\end{align*}
The right side is equal to
\begin{align*}
h\bigl(f \circ u_A\bigr) &= h\bigl(s \circ X^t \circ X^{T(a)} \circ u_A\bigr) \\
& = h\bigl(s \circ X^t \circ u_J \circ X^a\bigr) \\
& = h\bigl(s \circ (t_i)_{i \in I} \circ X^a\bigr) \\
& = M(s \circ (t_i)_{i \in I})(a).
\end{align*}
This shows that $M$ is, in fact, a model of $\L$. We have $C(M) = (M(X),h_M) = (A,h)$, since $M(X) = A$ holds by definition and $h_M : T(M(X)) \to M(X)$ equals $h$ since for $f \in T(A)$ we have
\[h_M(f) = M(f)(u_M) = h\bigl(f \circ X^{u_M}\bigr) = h(f).\]
Thus, $C$ is an isomorphism of categories.

\textbf{Correspondence of free models.} Since the diagram
\[\begin{tikzcd}[column sep=15pt]
\Mod(\L) \ar{rr}{C} \ar{dr} && \Alg(\T) \ar{dl} \\ & \Set & 
\end{tikzcd}\]
commutes, $C$ maps free $\L$-models onto free $\T$-algebras, and the free $\L$-models are just the representable $\L$-models by \cref{freemodel}.
\end{proof}

Let $\Monad$ denote the category of monads on $\Set$ with monad morphisms, which we can regard also as a $2$-category with identity $2$-morphisms.

\begin{thm} \label{lawmon}
There is an equivalence of $2$-categories $\Law_{\infty} \simeq \Monad$.
\end{thm}
 
\begin{proof}
The functor
\[\Law_{\infty} \to \Monad\]
is defined as follows: If $\L$ is an infinitary Lawvere theory, by the proof of \cref{ismonadic} we obtain a monad $\T=(T,\mu,\eta)$ on $\Set$ with $T(I) = \Hom_{\L}(X^I,X)$, $\eta_I(i)=\pr_i$ and $\mu_I(f)= f \circ u_I$. The action on morphisms is clear.

The functor
\[\Monad \to \Law_{\infty}\]
is defined as follows: If $\T$ is a monad on $\Set$, let $\K$ be its Kleisli category, which we see as the full subcategory of $\Alg(\T)$ of all free $\T$-algebras $F(I) = (T(I),\mu_I)$, where $I \in \Set$. Then $\K$ has coproducts, and every object is a coproduct of copies of the free $\T$-algebra $F(1)$ on one generator, so that its opposite $\K^{\op}$ is an infinitary Lawvere theory. The action on morphisms is clear.

First, we show that the composition
\[\Law_{\infty} \to \Monad \to \Law_{\infty}\]
is isomorphic to the identity: If $\L$ is an infinitary Lawvere theory, let $\T=(T,\mu,\eta)$ be its monad and $\K$ be its Kleisli category. Thus, the objects of $\K$ are the free $\T$-algebras with hom-sets
\[\Hom_{\K}(F(I),F(J)) \cong \Hom(I,T(J)) = \Hom(I,\Hom(X^J,X)) \cong \Hom(X^J,X^I).\]
Explicitly, a morphism $f : F(I) \to F(J)$ with underlying map $f : T(I) \to T(J)$ is mapped to $H(f) \coloneqq (f(\pr_i))_{i \in I} : X^J \to X^I$. For another morphism $g : F(J) \to F(K)$ with underlying map $g : T(J) \to T(K)$ we have $H(g \circ f) = H(f) \circ H(g)$, since for every $i \in I$
\begin{align*}
H(g \circ f)_i &= (g \circ f)(\pr_i) \\
& = \bigl(\mu_K \circ T(g \circ \eta_J) \circ f\bigr)(\pr_i) \\
&= \mu_K\bigl(T(g \circ \eta_J)(f(\pr_i))\bigr) \\
&= f(\pr_i) \circ X^{g \circ \eta_J} \circ u_K \\
& = f(\pr_i) \circ \bigl(g(\eta_J(j))\bigr)_{j \in J} \\
& = f(\pr_i) \circ (g(\pr_j))_{j \in J} \\
&= (H(f) \circ H(g))_i.
\end{align*}
Therefore, $F(I) \mapsto X^I$, $f \mapsto H(f)$ defines an equivalence $\K^{\op} \simeq \L$ mapping $F(1) \mapsto X$.

Finally, we show that the composition
\[\Monad \to \Law_{\infty} \to \Monad\]
is isomorphic to the identity: If $\T=(T,\eta,\mu)$ is a monad on $\Set$, $\K$ is its Kleisli category, then the monad $\S$ of the infinitary Lawvere theory $\K^{\op}$ has the underlying functor $S$ defined by
\[S(I) \coloneqq \Hom_{\K^{\op}}(X^I,X) = \Hom_{\K}(F(1),F(I)) \cong  \Hom(1,T(I)) \cong T(I).\]
Explicitly, the bijection $\alpha_I : S(I) \to T(I)$ maps a morphism of $\T$-algebras $f : F(1) \to F(I)$ with underlying map $f : T(1) \to T(I)$ to $(f \circ \eta_1)(\star) \in T(I)$, where $\star \in 1$ is the unique element. We need to check that $\alpha$ is an isomorphism of monads $\S \cong \T$. First, we have for every $i \in I$
\[\alpha_I(\eta_I(i))=\alpha_I(\pr_i) = \alpha_I\bigl(T(i : 1 \to I)\bigr) = T(i : 1 \to I)\bigl(\eta_1(\star)\bigr)=\eta_I\bigl((i : 1 \to I)(\star)\bigr) = \eta_I(i),\]
showing $ \alpha_I \circ \eta_I = \eta_I$. It remains to prove that the diagram
\[\begin{tikzcd}[column sep=35pt]
S^2(I) \ar[swap]{d}{\mu_I} \ar{r}{S(\alpha_I)} & S(T(I)) \ar{r}{\alpha_{T(I)}} & T^2(I) \ar{d}{\mu_I} \\
S(I) \ar[swap]{rr}{\alpha_I} && T(I) 
\end{tikzcd}\]
commutes. For every $f \in S^2(I) = \Hom(F(1),F(S(I))) \subseteq \Hom(T(1),T(S(I)))$ we have
\[\alpha_I(\mu_I(f)) = \alpha_I(u_I \circ f) = u_I(f(\eta_1(\star))),\]
and
\[\mu_I\bigl(\alpha_{T(I)}(S(\alpha_I)(f))\bigr) = \mu_I\bigl(\alpha_{T(I)}(T(\alpha_I) \circ f)\bigr) = \mu_I\bigl((T(\alpha_I) \circ f)(\eta_1(\star))\bigr),\]
so that it suffices to prove
\[u_I = \mu_I \circ T(\alpha_I).\]
Since $u_I$ is a morphism in $\K$ of free $\T$-algebras, this is equivalent to
\[u_I \circ \eta_{S(I)} = \alpha_I,\]
where both sides are maps $S(I) \to T(I)$. The morphism $u_I : F(S(I)) \to F(I)$ of $\T$-algebras is defined by the property that for all $g \in S(I) = \Hom(F(1),F(I))$ we have $u \circ \iota_g = g$, where $\iota_g : F(1) \to F(S(I))$ is the coproduct inclusion of index $g$, i.e.\ $\iota_g = T(g : 1 \to S(I))$. Combining this with the naturality of $\eta$ yields
\begin{align*}
u_I\bigl(\eta_{S(I)}(g)\bigr) &= u_I\bigl(\eta_{S(I)}\bigl((g : 1 \to S(I))(\star)\bigr)\bigr) \\
&= u_I\bigl(T(g : 1 \to S(I))(\eta_1(\star))\bigr) \\
& = g\bigl(\eta_1(\star)\bigr) = \alpha_I(g). \qedhere
\end{align*}
\end{proof}


\end{document}